\documentclass[11pt]{article}
\usepackage{amscd}
\usepackage{color}
\usepackage{amsfonts}
\usepackage[all]{xy}
\usepackage{amsmath,amssymb,amsthm,latexsym}
\usepackage{amscd}

\usepackage{bm}
\newtheorem{theorem}{Theorem}[section]

\newtheorem{definition}[theorem]{Definition}
\newtheorem{corollary}[theorem]{Corollary}
\newtheorem{lemma}[theorem]{Lemma}
\numberwithin{equation}{section}
\theoremstyle{remark}
\newtheorem{remark}[theorem]{Remark}
\newtheorem{example}[theorem]{\bf Example}
\newcommand{\R}{\mathbb{R}}
\newcommand{\C}{\mathbb{C}}
\newcommand{\D}{\mathbb{D}}

\newcommand{\FC}{\mathcal{C}}
\begin{document}
\title{\bf{On symmetric Willmore surfaces in spheres I: the orientation preserving case}}
\author{Josef Dorfmeister\footnote{Fakult\" at f\" ur Mathematik, TU-M\" unchen, Boltzmann str. 3, D-85747, Garching, Germany.
 dorfm@ma.tum.de \vspace{0.1mm}
} , Peng Wang \footnote{
  Department of Mathematics, Tongji University, Siping Road 1239, Shanghai, 200092, P. R. China.  \vspace{3mm}   \leftline{{netwangpeng@tongji.edu.cn}   \vspace{0.1mm} \hspace{95mm}}
   \vspace{1mm}  {\bf \ \ ~~Keywords:} Symmetry; Willmore surfaces;  Willmore sphere; invariant potentials.}
  }
\date{}
\maketitle

\begin{center}
{\bf Abstract}
\end{center}

In this paper we provide a systematic discussion of how to incorporate orientation preserving symmetries into the treatment of Willmore surfaces via the loop group method.

 In this context we first develop a general treatment of Willmore surfaces admitting  orientation preserving symmetries, and then show how to induce finite order rotational symmetries. We also prove, for the symmetric space which is the target space of the conformal Gauss map of Willmore surfaces in spheres, the  longstanding conjecture of the existence of meromorphic invariant potentials for the conformal Gauss maps of all compact Willmore surfaces in spheres. We also illustrate our results  by some concrete examples.

\section{Introduction}

 Surfaces with symmetries are always of interest, since they provide examples with nice behavior and frequently there also are some basic principles behind these examples.  Therefore, surfaces with (many)  symmetries are a frequently occurring topic in geometry papers. In \cite{DoWa1}, we began the study of Willmore surfaces by  the loop group method, a method already successfully used to investigate  other surface classes, like CMC surfaces in  $\R^3$, $\mathbb{H}^3$ and affine spheres (See for example \cite{Do;Osaka} and reference therein). It is therefore natural now to discuss  symmetries of Willmore surfaces in the framework of the loop group approach. The  study  of orientation preserving symmetries of Willmore surfaces is the main topic of this paper and the orientation reversing case as well as non-orientable Willmore surfaces will be treated in another paper.

A basic question certainly is what an effect a symmetry of a Willmore surface will have on the extended frames, the Maurer-Cartan forms, the corresponding conformal Gauss map, the holomorphic or meromorphic frames, and the potentials.

On the other hand, a Willmore surface $M$ with topology can be looked upon as an immersion from its universal covering $\tilde{M}$, with the fundamental group  $\pi_1(M)$  acting (invariantly) on the immersion. While this action
 $\pi_1(M)$ is trivial,  it turns out that  $\pi_1(M)$ acts generally non-trivially as a group of symmetries on each member of the associated family of the original immersion.

    We start by recalling some basic notation and some basic results concerning harmonic maps via loop group theory. This is the main contents of Section 2.  In Section 3 we start the discussion of Willmore surfaces with symmetries. First we show that in general a symmetry of the image of an immersion will induce a symmetry of the surface. This yields four kinds of possibilities.
Here we will consider orientation preserving symmetries.

In Section 4, we introduce the notion of a monodromy (loop) matrix and determine how it enters the transformation formulas for the extended frames, the holomorphic/meromorphic frames and the associated families of Willmore immersions. It turns out that the transformation formulas also involve some gauges. In particular, the transformation formula for the potentials
( differential one-forms generating all Willmore immersions  in the loop group formalism) only involves gauges.

Most natural symmetries of a surface are finite order ``rotations''. We thus consider as first applications of our general theory  symmetries of finite order. This type of symmetry of a Willmore surface is characterized in Section 4.2 by a very simple transformation behavior of its potential. We end Section 4 with some examples, illustrating our results.

In section 5 we discuss the fundamental group as a group of symmetries.
It turns out that if we consider a Willmore immersion from a Riemann surface $M$ into $S^{n+2}$, then such a Willmore surface can be generated from an invariant potential. The potential can be chosen to be holomorphic, if $M$ is non-compact and will be meromorphic if $M$ is compact (of any genus). Thus to generate a Willmore immersion of $M$ one can start from some holomorphic/meromorphic differential one-form on $M$ and pull it back to the universal cover $\tilde{M}$. The only additional property still needed  to obtain an immersion defined on $M$ is a closing condition
of the monodromy matrices, which is usually only satisfied for a finite number of values of the loop parameter. In particular, the fundamental group  $\pi_1(M)$ will actually generally act as a non-trivial group of symmetries on almost all surfaces of the associated family. (The details can be found in section 5.)

 In Section 6 we provide the proof  the existence of invariant meromorphic potentials for Willmore immersions from compact Riemann surfaces to $S^{n+2}$.  This solves a longstanding conjecture.
The proof for the non-compact case can be taken almost verbatim from \cite{Do-Ha5}.

We end this paper by considering Willmore surfaces $f : \D \rightarrow S^{n+2}$ which induce complete metrics. We show that in this case, for  every
symmetry $R$ of $f$, there exists a conformal automorphism $\gamma$ of $\D$ such that $f(\gamma.z) = Rf(z)$ holds for all $z\in \D$.

{\it Throughout this paper, by a ''surface'' we mean a ''branched surface''
unless stated explicitly otherwise.}

\section{Review of basic  notation and basic results}

In \cite{DoWa1} we have presented a detailed discussion of the basic loop group approach to the theory of Willmore surfaces in spheres.
In this introductory section we recall some notation and some results and refer for more details to \cite{DoWa1}.

 Let $G$ be the connected, real, semi-simple non-compact matrix Lie group $G = SO^+(1,n+3)$. Let $G/K$ be the inner symmetric space defined by  the involution $\sigma: G\rightarrow G$, given by $\sigma = Ad\dot{S}$, where
 $\dot{S} = diag(-I_4, I_n) $ and $K = Fix(G)^{\sigma} = SO^+ (1,3) \times SO(n)$, a connected, real, semi-simple subgroup of $SO^+(1,n+3)$.
Note that $G/K$ carries a left-invariant non-degenerate symmetric bilinear form, derived from a bi-invariant metric (the Killing form) on $G$.

Let $\mathfrak{g} = \mathfrak{so}(1,n+3)$ and
$\mathfrak{k} = \mathfrak{so}(1,3) \times \mathfrak{so}(n)$  denote the
Lie algebras of $G$ and $K$ respectively. The involution $\sigma$ induces
a decomposition of $\mathfrak{g}$ into eigenspaces,  the (generalized) Cartan decomposition
$$\mathfrak{g}=\mathfrak{k}\oplus\mathfrak{p},\hspace{5mm} \hbox{ with }\  [\mathfrak{k},\mathfrak{k}]\subset\mathfrak{k},
~~~ [\mathfrak{k},\mathfrak{p}]\subset\mathfrak{p}, ~~~
[\mathfrak{p},\mathfrak{p}]\subset\mathfrak{k}.$$
Let $\pi:G\rightarrow G/K$ denote the projection of $G$ onto $G/K$.

Now let  $\mathfrak{g^{\mathbb{C}}} = \mathfrak{so}(1,n+3,\C)$ be the complexification of $\mathfrak{g}$ and $G^{\mathbb{C}} = SO(1,n+3,\C)$  the connected complex (matrix) Lie group with Lie algebra $\mathfrak{g^{\mathbb{C}}}$.
Let $\tau$ denote the complex anti-holomorphic involution
$g \rightarrow \bar{g}$ of $G^{\mathbb{C}}$.
Then $G=Fix^{\tau}(G^{\C})^0$, where $H^\circ$ denotes the identity component of the group $H$.
The inner involution $\sigma: G \rightarrow G $ commutes with the complex conjugation $\tau$ and extends to
the complexified Lie group $G^\C$, $\sigma: G^{\mathbb{C}}\rightarrow G^{\mathbb{C}}$. Then  $K^{\mathbb{C}} = \hbox{Fix}^{\sigma}(G^{\mathbb{C}})^0$ is the smallest (connected) complex subgroup of $G^{\C}$ containing $K$. The Lie algebra of $K^{\mathbb{C}}$ is
$\mathfrak{k^{\mathbb{C}}} = \mathfrak{so}(1,3,\C) \times \mathfrak{so}(n, \C) $.


\subsection{Harmonic maps into symmetric spaces}

Let $G/K$ be the symmetric space above. Let $\mathcal{F}:M\rightarrow G/K$ be a harmonic map from a connected Riemann surface $M$.
Let $U\subset M$ be an open contractible subset.
Then there exists a frame $F: U\rightarrow G$ such that $\mathcal{F}=\pi\circ F$.
Let $\alpha$ denote the Maurer-Cartan form of $F$. Then $\alpha$ satisfies the Maurer-Cartan equation
and altogether we have
\begin{equation*}F^{-1}d F= \alpha, \hspace{5mm} d \alpha+\frac{1}{2}[\alpha\wedge\alpha]=0.
\end{equation*}
Decomposing $\alpha$ with respect to $\mathfrak{g}=\mathfrak{k}\oplus\mathfrak{p}$ we obtain
$$\alpha=\alpha_{ \mathfrak{k}  } +\alpha_{ \mathfrak{p} }, \
\alpha_{\mathfrak{k  }}\in \Gamma(\mathfrak{k}\otimes T^*M),\
\alpha_{ \mathfrak{p }}\in \Gamma(\mathfrak{p}\otimes T^*M).$$ Moreover, considering the complexification $TM^{\mathbb{C}}=T'M\oplus T''M$, we decompose $\alpha_{\mathfrak{p}}$ further into the $(1,0)-$part $\alpha_{\mathfrak{p}}'$ and the $(0,1)-$part $\alpha_{\mathfrak{p}}''$. Set
 \begin{equation}\label{spec-par}
\alpha_{\lambda}=
\lambda^{-1}\alpha_{\mathfrak{p}}'+\alpha_{\mathfrak{k}}+
\lambda\alpha_{\mathfrak{p}}'', \hspace{5mm}  \lambda\in S^1.
\end{equation}

\begin{lemma} $($\cite{DPW}$)$ The map  $\mathcal{F}:M\rightarrow G/K$ is harmonic if and only if
\begin{equation}\label{integr}d
\alpha_{\lambda}+\frac{1}{2}[\alpha_{\lambda}\wedge\alpha_{\lambda}]=0,\ \ \hbox{for all}\ \lambda \in S^1.
\end{equation}
\end{lemma}

\begin{definition} Let $\mathcal{F}:M\rightarrow G/K$ be harmonic
and $\alpha_{\lambda}$ the differential one-form defined above. Since, by the lemma,
$\alpha_{\lambda}$ satisfies the integrability condition \eqref{integr}, we consider
on $U \subset M$
the solution $F(z,\lambda)$, to the equation
$$d F(z,\lambda)= F(z, \lambda)\alpha_{\lambda}$$
with the initial condition $F(z_0,\lambda)=F(z_0)\in K$,
where $z_0 \in U$ is a fixed base point. The map $F(z, \lambda)$
is called the {\em extended frame} of the harmonic map $\mathcal{F}$ normalized at the base point $z=z_0$. Note that it satisfies $F(z,\lambda =1)=F(z)$.
 \end{definition}

\subsection{Loop groups, decomposition theorems, and the loop group method}

For the construction of (new) Willmore surfaces in spheres we will employ the loop group method. In this context we consider the twisted loop groups of $G$ and $G^{\mathbb{C}}$
and some of their frequently occurring subgroups:
\begin{equation*}
\begin{array}{llll}
\Lambda G^{\mathbb{C}}_{\sigma} ~&=\{\gamma:S^1\rightarrow G^{\mathbb{C}}~|~ ,\
\sigma \gamma(\lambda)=\gamma(-\lambda),\lambda\in S^1  \},\\[1mm]
\Lambda G_{\sigma} ~&=\{\gamma\in \Lambda G^{\mathbb{C}}_{\sigma}
|~ \gamma(\lambda)\in G, \hbox{for all}\ \lambda\in S^1 \},\\[1mm]
\Omega G_{\sigma} ~&=\{\gamma\in \Lambda G_{\sigma}|~ \gamma(1)=e \},\\[1mm]
\Lambda_{*}^{-} G^{\mathbb{C}}_{\sigma} ~&=\{\gamma\in \Lambda G^{\mathbb{C}}_{\sigma}~
|~ \gamma \hbox{ extends holomorphically to }D_{\infty},\  \gamma(\infty)=e \},\\[1mm]
\Lambda^{+} G^{\mathbb{C}}_{\sigma} ~&=\{\gamma\in \Lambda G^{\mathbb{C}}_{\sigma}~
|~ \gamma \hbox{ extends holomorphically to }D_{0} \},\\[1mm]
\Lambda_{S}^{+} G^{\mathbb{C}}_{\sigma} ~&=\{\gamma\in
\Lambda G^{\mathbb{C}}_{\sigma}~|~   \gamma(0)\in S \},\\[1mm]
\end{array}\end{equation*}
where $D_0=\{z\in \mathbb{C}| \ |z|<1\}$, $D_\infty=\{z\in \mathbb{C}| \ |z|>1\}$ and $S$ is some subgroup of $K^\C$.
\vspace{2mm}

In \cite{DoWa1} we have shown that there exists some closed, connected, solvable Lie subgroup $S$ of $K^\C$ such that $K \times S \rightarrow K^\C$ is a diffeomorphism onto the open subset $K\cdot S$ of $K^\C$.

If the group $S$ is chosen to be $S = (K^\C)^0$, then we write $\Lambda_{\mathcal{C}}^{\pm} G^{\mathbb{C}}_{\sigma} $.

We frequently use the following decomposition theorems.

\begin{theorem} \label{thm-decomposition}(\cite{DPW}, \cite{PS}, \cite{DoWa1}.)

{\em (i)(Iwasawa decomposition)}

$(1) \hspace{2mm} (\Lambda G^{\C})_{\sigma} ^0=
\bigcup_{\delta \in \Xi }( \Lambda G)_{\sigma}^0\cdot \delta\cdot
\Lambda^{+} G^{\mathbb{C}}_{\sigma},$

$(2)$ There exist exactly two open Iwasawa cells in the connected loop group
$(\Lambda G^{\mathbb{C}}_{\sigma})^0$,  one given by $\delta = e$ and the other one by $\delta = diag(-1,1,1,1,-1,1,1,...,1) $.

$(3)$
There exists a closed, connected solvable subgroup $S \subseteq K^\C$ such that
the multiplication $\Lambda G_{\sigma}^0 \times \Lambda^{+}_S G^{\mathbb{C}}_{\sigma}\rightarrow
\Lambda G^{\mathbb{C}}_{\sigma}$ is a real analytic diffeomorphism onto the open subset
$ \Lambda G_{\sigma}^0 \cdot \Lambda^{+}_S G^{\mathbb{C}}_{\sigma}      \subseteq \mathcal{I}^{\mathcal{U}}_e \subset(\Lambda G^{\mathbb{C}}_{\sigma})^0$.

{\em (ii)(Birkhoff decomposition)}

(1) $(\Lambda {G}^\C )^0= \bigcup \Lambda^{-}_{\mathcal{C}} {G}^{\mathbb{C}}_{\sigma} \cdot \omega \cdot \Lambda^{+}_{\mathcal{C}} {G}^{\mathbb{C}}_{\sigma}$
where the $\omega$'s are representatives of the double cosets .

(2) The multiplication $\Lambda_{*}^{-} {G}^{\mathbb{C}}_{\sigma}\times
\Lambda^{+}_\FC {G}^{\mathbb{C}}_{\sigma}\rightarrow
\Lambda {G}^{\mathbb{C}}_{\sigma}$ is an analytic  diffeomorphism onto the
open and dense subset $\Lambda_{*}^{-} {G}^{\mathbb{C}}_{\sigma}\cdot
\Lambda^{+}_\FC {G}^{\mathbb{C}}_{\sigma}$ {\em ( big Birkhoff cell )}.

\end{theorem}
For more details we refer to \cite{DoWa1} and \cite{Ke1}.

Loops which have a finite Fourier expansion are called {\it algebraic loops} and
 denoted by the subscript $``alg"$, like $\Lambda_{alg} G_{\sigma},\ \Lambda_{alg} G^{\mathbb{C}}_{\sigma},\
\Omega_{alg} G_{\sigma} $ (see also \cite{BuGu}, \cite{Gu2002} and \cite{DoWa2}). And we define
  \begin{equation}\label{eq-alg-loop}\Omega^k_{alg} G_{\sigma}:=\{\gamma\in
\Omega_{alg} G_{\sigma}|
Ad(\gamma)=\sum_{|j|\leq k}\lambda^jT_j \}\subset \Omega_{alg} G_{\sigma} .\end{equation}

 With the loop group decompositions as stated above, we obtain a
construction scheme of harmonic maps from a surface into $G/K$.

\begin{theorem}\label{thm-DPW}\cite{DPW}, \cite{DoWa1}, \cite{Wu}.
Let $\D$ be a contractible open subset of $\C$ and $z_0 \in \D$ a base point.
Let $\mathcal{F}: \D \rightarrow G/K$ be a harmonic map with $\mathcal{F}(z_0)=eK.$
Then the associated family  $\mathcal{F}_{\lambda}$ of $\tilde{F}$ can be lifted to a map
$F:\D \rightarrow \Lambda G_{\sigma}$, the extended frame of $\mathcal{F}$, and we can assume  w.l.g. that $F(z_0,\lambda)= e$ holds.
Under this assumption,

(1) The map $F$ takes only values in
$ \mathcal{I}^{\mathcal{U}}\subset \Lambda G^{\mathbb{C}}_{\sigma}$.

 (2) There exists a discrete subset $\D_0\subset \D$ such that on $\D\setminus \D_0$
we have the decomposition
$$F(z,\lambda)=F_-(z,\lambda)  F_+(z,\lambda),
$$where $$ F_-(z,\lambda)\in\Lambda_{*}^{-} G^{\mathbb{C}}_{\sigma}
\hspace{2mm} \mbox{and} \hspace{2mm} F_+(z,\lambda)\in \Lambda^{+}_{\FC} G^{\mathbb{C}}_{\sigma}.$$
and $F_-(z,\lambda)$ is meromorphic in $z \in \D$.

Moreover,
$$\eta= F_-(z,\lambda)^{-1} d F_-(z,\lambda)$$
is a $\lambda^{-1}\cdot\mathfrak{p}^{\mathbb{C}}-\hbox{valued}$ meromorphic $(1,0)-$
form with poles at points of $\D_0$ only.

$(3)$ Spelling out the converse procedure in detail we obtain:
Let $\eta$ be a  $\lambda^{-1}\cdot\mathfrak{p}^{\mathbb{C}}-\hbox{valued}$ meromorphic $(1,0)-$ form for which the solution
to the ODE
\begin{equation}
F_-(z,\lambda)^{-1} d F_-(z,\lambda)=\eta, \hspace{5mm} F_-(z_0,\lambda)=e,
\end{equation}
is meromorphic on $\D$, with  $\D_0$ as set of possible poles.
Then on the open set $\D_{\mathcal{I}} = \{  z \in \D\setminus {\D_0}; F(z,\lambda)
\in \mathcal{I}^{\mathcal{U}}\}$ we
define locally $\tilde{F}(z,\lambda)$ via the factorization
 $\mathcal{I}^{\mathcal{U}} =  ( \Lambda G_{\sigma})_o \cdot \Lambda_S^{+} G^{\mathbb{C}}_{\sigma}
\subset  \Lambda G^{\mathbb{C}}_{\sigma}$:
\begin {equation}\label{Iwa}
F_-(z,\lambda)=\tilde{F}(z,\lambda)  \tilde{F}_+(z,\lambda)^{-1}.
\end{equation}
 This way one obtains locally an extended frame
$$ \tilde{F}(z,\lambda)=F_-(z,\lambda)  \tilde{F}_+(z,\lambda)\\$$
of some harmonic map from $ \D_{\mathcal{I}}$ to $G/K$.

$(4) $ Any harmonic map  $\mathcal{F}: \D\rightarrow G/K$ can be derived from a
$\lambda^{-1}\cdot\mathfrak{p}^{\mathbb{C}}-\hbox{valued}$ meromorphic $(1,0)-$ form $\eta$ on $\D$.
Moreover, the two constructions outlined above  are inverse to each other (on appropriate domains of definition), if normalizations at some base point are used.

\end{theorem}

\begin{remark}
 The restriction above to ``local" factorizations has two reasons: one is that there may be singularities due to the poles in the potentials and the second is that the factorization  (\ref{Iwa}) will not be unique in general. In the latter case the ambiguity is removed when descending to the harmonic map.
Unfortunately, there is yet another reason for a singularity: it will occur, when
the solution $F_-(z,\lambda)$ leaves the open set $\mathcal{I}^{\mathcal{U}}$. Also in this case it can happen that the associated harmonic map is non-singular. But in general singularities will remain.
Therefore, for concrete examples one needs to make sure that no singularities occur.

\end{remark}

\begin{definition}\cite{DPW},\ \cite{Wu}.
The $\lambda^{-1}\cdot \mathfrak{p}^{\mathbb{C}}-\hbox{valued}$ meromorphic $(1,0)$
form $\eta$ is called the {\em normalized potential} for the harmonic
map $\mathcal{F}$ with the point $z_0$ as the reference point. And $F_-(z,\lambda)$ given above is called the {meromorphic extended frame}.
\end{definition}

The normalized potential is uniquely determined, if some base point on $M$ is fixed and the frames are normalized to $e$ at the base point.
The normalized potential is usually meromorphic in $z$.

In many applications it is much more convenient to use potentials which have a Fourier expansion containing more than one power of $\lambda$.
And when permitting many (maybe infinitely many) powers of $\lambda$,  one can even frequently obtain holomorphic coefficients.

\begin{theorem}\cite{DPW}, \cite{DoWa1}.  Let $\D$ be a contractible open subset of $\C$.
Let $F(z,\lambda)$ be the frame of some harmonic map
into $G/K$. Then there exists some $V_+ \in \Lambda^{+} G^{\mathbb{C}}_{\sigma} $ such that $C(z,\lambda) =
F V_+$ is holomorphic in $z\in\mathbb{D}$ and in $\lambda \in \mathbb{C}^*$.
Then the Maurer-Cartan form $\eta = C^{-1} dC$ of $C$ is a holomorphic $(1,0)-$form on $\D$ and it is easy to verify that $\lambda \eta$ is holomorphic for $\lambda \in \C$.

Conversely, Let $\eta\in\Lambda\mathfrak{g}^{\C}_{\sigma}$ be a holomorphic  $(1,0)-$ form such that $\lambda \eta$ is holomorphic for $\lambda\in\C$, then by the same steps as in Theorem \ref{thm-DPW} we obtain a harmonic map $\mathcal{F}: \D\rightarrow G/K$.
\end{theorem}

The matrix function $C(z,\lambda)$ associated with a holomorphic $(1,0)-$ form $\eta$ as in the theorem will be called the {\em holomorphic extended frame} of the harmonic map $\mathcal{F}$.\\

In \cite{Wang-S4}, an explicit formula is obtained for a generic isotropic Willmore surface in $S^4$. We state the result below and use it to produce the examples presented in the following section.

\begin{theorem}\label{th-willmore-iso-formula}\cite{Wang-S4} Let
$$\eta=\lambda^{-1}\left(
                     \begin{array}{cc}
                       0 & \hat{B}_1\\
                       -\hat{B_1}I_{1,3} & 0 \\
                     \end{array}
                   \right)
$$
with
\begin{equation}\label{eq-b1}
\hat{B}_1=\frac{1}{2}\left(
                     \begin{array}{cccc}
                      i(f_3'-f_2')&  -(f_3'-f_2') \\
                      i(f_3'+f_2')&  -(f_3'+f_2')   \\
                      f_4'-f_1' & i(f_4'-f_1')   \\
                      i(f_4'+f_1') & -(f_4'+f_1')  \\
                     \end{array}
                   \right),\  f_1'f_4'+f_2'f_3'=0.
                   \end{equation}
                   Then the corresponding Willmore surface
is of the form
\begin{equation}\label{eq-willmore in s4-y-1}
\begin{split}
Y_1=&|f_1'|^2\left(\begin{array}{cc}
(1+|f_2|^2+|f_4|^2)\\
1-|f_2|^2+|f_4|^2\\
-i (-\bar{f}_2f_4+f_2\bar{f}_4)\\
-(\bar{f}_2f_4+f_2\bar{f}_4) \\
i (\bar{f}_2-f_2) \\
(\bar{f}_2+f_2) \\
\end{array}\right)+|f_2'|^2\left(\begin{array}{cc}
(1+|f_1|^2+|f_3|^2)\\
-(1+|f_1|^2-|f_3|^2)\\
i (-\bar{f}_1f_3+f_1\bar{f}_3)\\
\bar{f}_1f_3+f_1\bar{f}_3 \\
i (f_3-\bar{f}_3) \\
-(f_3+\bar{f}_3) \\
\end{array}\right)\\
&~~~~+f_1'\bar{f}_2'\left(\begin{array}{cc}
-\bar{f}_1f_2+\bar{f}_3f_4\\
\bar{f}_1f_2+\bar{f}_3f_4\\
-i (1+\bar{f}_1f_4+f_2\bar{f}_3)\\
-(1-\bar{f}_1f_4+f_2\bar{f}_3)\\
i (-\bar{f}_1+f_4) \\
-(\bar{f}_1+f_4) \\
\end{array}\right)+ \bar{f}_1'f_2'
\overline{\left(\begin{array}{cc}
-\bar{f}_1f_2+\bar{f}_3f_4\\
\bar{f}_1f_2+\bar{f}_3f_4\\
-i (1+\bar{f}_1f_4+f_2\bar{f}_3)\\
-(1-\bar{f}_1f_4+f_2\bar{f}_3)\\
i (-\bar{f}_1+f_4) \\
-(\bar{f}_1+f_4) \\
\end{array}\right)}.\\
\end{split}
\end{equation}
and
\begin{equation}
Y_{\lambda}=\left(
    \begin{array}{ccccccc}
        1 & 0& 0 & 0 & 0 & 0 \\
        0 & 1& 0 & 0 & 0 & 0 \\
        0 & 0& 1 & 0& 0 & 0 \\
        0 & 0& 0 & 1& 0 & 0  \\
       0 & 0 & 0 & 0& \frac{\lambda+\lambda^{-1}}{2} &  \frac{\lambda-\lambda^{-1}}{-2i} \\
        0 & 0& 0 & 0&  \frac{\lambda-\lambda^{-1}}{2i} &  \frac{\lambda+\lambda^{-1}}{2}  \\
    \end{array}
  \right)
\cdot Y_1
\end{equation}
when $f_1'f_2'\not\equiv0$ .
\end{theorem}


\section{Willmore immersions with symmetries}

Let $y: M \rightarrow S^{n+2}$ be a Willmore immersion from a Riemann surface $M$ to $S^{n+2}$ (Here we follow the treatment in \cite{DoWa1}).
Assume that $R$ is a conformal transformation of $S^{n+2}$, considered as a linear transformation of  $\mathbb{R}^{n+4}_1$, i.e., $R\in O^+(1,n+3)$.
Assume $R(y(M )) = y(M )$. Then $R$  expresses some  ``symmetry" of $y$ and $M$.

\begin{theorem}\label{thm-induce-sym} Let $M$ be a simply connected Riemann surface and $y: M \rightarrow S^{n+2}$  a Willmore surface. Let  $R$ be a conformal transformation of $S^{n+2}$ satisfying $R y(M) = y(M)$
and preserving the orientation of $y$. Assume moreover, that
the metric induced on $M$ by $y$ is complete, then  there exists some conformal transformation $\gamma$ of $M$ with respect to the induced metric such that
\begin{equation}\label{eq-sym-y}
y( \gamma .p) =[Y(\gamma.p)]=[RY(p)]= Ry(p)\ \hbox{ for all }\ p \in M.
\end{equation}
Moreover, $\gamma$ is a holomorphic automorphism of $M$ if $R$ induces an orientation preserving symmetry of $y$.
$\gamma$ is an anti-holomorphic automorphism of $M$ if $R$ induces an orientation preserving symmetry of $y$.
\end{theorem}
\begin{proof} See Section 7.
\end{proof}

From Theorem \ref{thm-induce-sym}, since $\det R$ can be either $1$ or $-1$, there are four cases concerning the symmetry $(R,\gamma)$:

1. $\det R=1$ and $\gamma$ is a holomorphic automorphism of $M$.

2. $\det R=-1$ and $\gamma$ is a holomorphic automorphism of $M$.

3. $\det R=1$ and $\gamma$ is an anti-holomorphic automorphism of $M$.

4. $\det R=-1$ and $\gamma$ is an anti-holomorphic automorphism of $M$.

In the next section, we will deduce a general theory for the first two cases, i.e, the Willmore immersions with orientation preserving symmetries.
For the  Willmore immersions with orientation reversing symmetries, including non-orientable Willmore surfaces, we will deal with  in the second paper.

\section{Willmore immersions with orientation preserving symmetries}

\subsection{The general case}

Among all Willmore immersions into $S^{n+2}$, those with symmetries are of interest in many cases.  Also fundamental groups induce symmetries:
if $y:M \rightarrow S^{n+2}$ is a Willmore immersion and $\pi_1(M) \neq \{1 \}$, then $\pi_1(M)$ acts as a group of symmetries on all $y_\lambda$, and in general this action is non-trivial.

In this subsection we will briefly outline the general theory of symmetries of Willmore immersions and in the following sections we will present some applications.

In view of Theorem \ref{thm-induce-sym}, if we talk about an orientation preserving symmetry of some Willmore immersion we always mean a pair $\left(\gamma , R \right)$ as above. Actually, it is not necessary here to assume that $M$ is simply connected nor that the induced metric is complete. However, in most cases it is convenient to consider the universal cover $\tilde{M}$ of $M$.

Moreover, by \eqref{eq-sym-y}, there exists some function $\varsigma=\varsigma(p)$ such that
$$ Y(\gamma.p)=\varsigma \cdot RY(p).$$
Therefore,
$$\{Y(\gamma.p), Y(\gamma.p)_z, Y(\gamma.p)_{\bar{z}}, Y(\gamma.p)_{z\bar{z}}\}\subset\hbox{Spanc}_{\C}\{ RY(p), RY(p)_z, RY(p)_{\bar{z}}, RY(p)_{z\bar{z}}\}.$$
Together with the definition of the conformal Gauss map (See Section 2 of \cite{DoWa1}), we obtain
\begin{equation}
f(\gamma (p)) = \hat{R} f(p).
\end{equation}

On the other hand, there exists a special kind of Willmore surfaces named S-Willmore surfaces by  Ejiri. As stated in  Definition 2.8 of \cite{DoWa1} we use the equivalent definition that S-Willmore surfaces are exactly those Willmore surfaces which admit  a dual (Willmore) surface (\cite{Bryant1984}, \cite{Ejiri1988}, \cite{Ma2005}). Minimal surfaces in Riemannian space forms provide standard examples of S-Willmore surfaces. For an S-Willmore surface $y$, when its dual  $\hat{y}$ is an immersion, then, since $y$ and $\hat{y}$ share the same conformal Gauss map, it is possible that there exists some
$R\in O^+(1,n+3)$ such that
\begin{equation}\label{eq-sym-dual}
\hat{Y}(\gamma(p)) =  R Y(p),
\end{equation}
for some holomorphic automorphism $\gamma:M\rightarrow M$.
 In this case,  on the conformal Gauss map level we obtain
\begin{equation}\label{eq-sym}
f(\gamma (p))=\hat{f}(\gamma (p)) = \hat{R} f(p).
\end{equation}
Note that in this case the symmetry of the conformal Gauss map does not stem from a symmetry of some  Willmore immersion.

Altogether we have proven the following

\begin{theorem}\label{th-symm-f} Let $y: M \rightarrow  S^{n+2}$ be a Willmore immersion and $\left(\gamma , R\right)$ a symmetry of $y$ as in \eqref{eq-sym-y} or let $R$ be a map relating $y$ to  its dual surface $\hat{y}$ as stated in \eqref{eq-sym-dual}.
Then the conformal Gauss map $f$ of $y$ satisfies
\begin{equation}\label{eq-sym}
f(\gamma (p)) = \hat{R} f(p)
\end{equation}
where $\hat{R}$ denotes the isometry of the symmetric space $SO^+(1,n+3)/ SO^+(1,3)\times SO(n)$,
 which is induced by $R$ considered as an element of the Lorentz group $O^+(1,n+3)$.
\end{theorem}

Symmetries have been investigated for many types of harmonic maps.
Next we consider symmetries defined on contractible domains. The case of $S^2$ and more general surfaces will be considered at the end of this section.
In our case we obtain

\begin{theorem}
 Let $\D$ be contractible and let $f : \D \rightarrow SO^+(1,n+3)/ SO^+(1,3)\times SO(n)$ be a harmonic map
with symmetry $\left( \gamma , \hat{R}\right) $.
Let $F$ denote the moving frame associated with $f$ as in  Section 2 of \cite{DoWa1}.  Then there exists some  $k: \D \rightarrow SO^+(1,3)\times O(n)$ such that
\begin{equation}
\gamma^* F (p) = R F(p) k(p),
\end{equation}
Moreover $\det k(p)=\det R=\pm 1$.
\end{theorem}
\begin{proof}
This follows directly from \eqref{eq-sym}. The statement that $k(p)$ takes value in $SO^+(1,3)\times O(n)$ comes from the orientation preserving
property of $\gamma$.
\end{proof}

\begin{remark}
Above and in many cases we do not choose a specific frame $F$. Therefore we cannot say much about $k(z,\bar{z})$.
However, if we choose the frame as in \cite{DoWa1},  Proposition 2.2, then one can compute the $SO^+(1,3)-$part of $k$ explicitly.

To be concrete, we note first
$\gamma^*Y=e^{\tau}RY$ for some function $\tau$. So we obtain
$$\gamma^* (Y_z)=\left(\frac{d\gamma}{dz}\right)^{-1}(\tau_ze^{\tau}RY+e^{\tau}RY_z),$$
from which we infer
$\frac{d\gamma}{dz}=e^{\tau+i\theta}$ for some real function $\theta$, and
$$\gamma^*e_1=Re_1\cos\theta-Re_2\sin\theta+ aRY, \ \gamma^*e_2=Re_1\sin\theta+Re_2\cos\theta+ bRY,$$ $$\gamma^*N=e^{-\tau}(RN+a\gamma^*e_1+b\gamma^*e_2+(a^2+b^2)RY),$$
with $e_1=Y_z+Y_{\bar{z}}$, $e_1=-i(Y_z-Y_{\bar{z}})$ and $a-ib=2e^{-i\theta}\tau_z.$
So we obtain
\begin{equation}\label{eq-F-k}
  \gamma^*F(z,\bar{z})=F(z,\bar{z})k, \hspace{3mm} k=\left(
                    \begin{array}{cc}
                      k_1 & 0 \\
                      0 & k_2 \\
                    \end{array}
                  \right),
\end{equation}
with
\begin{equation}
\begin{split} k_1=\left(
                                   \begin{array}{cccc}
                                   \frac{ e^{\tau}+  e^{-\tau}(a^2+b^2)+ e^{-\tau} }{2}  &\frac{ -e^{\tau}+  e^{-\tau}(a^2+b^2)+ e^{-\tau} }{2}   & \frac{a}{\sqrt{2}} &  \frac{b}{\sqrt{2}}  \\
                                   \frac{ -e^{\tau}-  e^{-\tau}(a^2+b^2)+ e^{-\tau} }{2}  &\frac{ e^{\tau}- e^{-\tau}(a^2+b^2)+ e^{-\tau} }{2}   & -\frac{a}{\sqrt{2}} & - \frac{b}{\sqrt{2}}  \\
                                     \frac{e^{-\tau}(a\cos\theta+b\sin\theta)}{\sqrt{2}} &    \frac{e^{-\tau}(a\cos\theta+b\sin\theta)}{\sqrt{2}}& \cos\theta & \sin\theta \\
                                       \frac{e^{-\tau}(b\cos\theta-a\sin\theta)}{\sqrt{2}} &    \frac{e^{-\tau}(b\cos\theta-a\sin\theta)}{\sqrt{2}}  & -\sin\theta & \cos\theta \\
                                   \end{array}
                                 \right).
\end{split}\end{equation}
\ \\
\end{remark}

Note that the symmetry equation for $F$ above implies
\begin{equation}
\gamma^* \alpha = k^{-1} \alpha k + k^{-1} dk.
\end{equation}
So
$$\gamma^* \alpha_{\mathfrak{p}} '=  k^{-1} \alpha_{\mathfrak{p}} ' k,\ \gamma^* \alpha_{\mathfrak{p}} ''= k^{-1} \alpha_{\mathfrak{p}} '' k,\ \gamma^* \alpha_{\mathfrak{k}} = k^{-1} \alpha_{\mathfrak{k}} k + k^{-1} dk.$$
Recalling how the spectral parameter was
introduced in (\ref{spec-par}), it follows

\begin{equation} \label{symmetryalpha}
\gamma^* \alpha_{ \lambda } =
k^{-1} \alpha_{ \lambda }  k + k^{-1} dk.
\end{equation}

From this we obtain (see also e.g. \cite{Do-Ha2}, \cite{Do-Ha3})
\begin{theorem}\label{th-symm-1}
Let $\D$ be contractible and $(\gamma , \hat{R})$
a symmetry of the harmonic map $f: \D \rightarrow SO^+(1,n+3)/ SO^+(1,3)\times SO(n)$, where $\hat{R}$ is induced by $R\in O^+(1,n+3)$.
Let $F(z,\bar{z},\lambda )$ be a lift of $f$ satisfying the initial condition $F(0,0,\lambda)=I$. Then  there exists some $\chi ( \lambda ) $
such that
\begin{equation} \label{symmetryFlambda}
\gamma^* F(z,\bar{z},\lambda ) = \chi(\lambda) F(z,\bar{z},\lambda )k(z,\bar{z}),
\end{equation}
where  $k: \D \rightarrow SO^+(1,3)\times O(n)$ is independent of $\lambda$. Moreover, $\chi\in (\Lambda SO^+(1,n+3)_{\sigma})^0$ when $\det R=1$, and $\chi \tilde{P}\in (\Lambda SO^+(1,n+3)_{\sigma})^0$ with $\tilde{P}=\hbox{diag}(1,1,1,1,-1,1,\cdots,1)$ when $\det R=-1$.
\end{theorem}

\begin{proof} Since $\D$ is contractible, the  system of partial differential equations
$$dF(z,\bar{z},\lambda ) =F(z,\bar{z},\lambda ) \alpha_{ \lambda} $$
 has a solution on all of $\D$. By \eqref{symmetryalpha} it is clear that
$\gamma^* F(z,\bar{z},\lambda ) $ and $F(z,\bar{z},\lambda )k(z,\bar{z}) $ solve the same PDE-system.
Hence there exists some $\chi$ such that \eqref{symmetryFlambda} holds.
The right factor of equation \eqref{symmetryFlambda} is obviously the same as for the case $\lambda =1$.
Since we have normalized $F(z,\bar{z},\lambda ) $ to satisfy $F(0,0,\lambda)=I$, equation \eqref{symmetryFlambda} implies $\chi(\lambda)=F(\gamma.(0,0),\lambda)k(0,0)^{-1}$. In particular,
$\chi(\lambda)\in(\Lambda SO^+(1,n+3)_{\sigma})^0$  when $\det R=1$, and $\chi\cdot \tilde{P}\in (\Lambda SO^+(1,n+3)_{\sigma})^0$ with $\tilde{P}=\hbox{diag}(1,1,1,1,-1,1,\cdots,1)$ when $\det R=-1$.
\end{proof}

The matrix $\chi$ will be called a {\bf monodromy matrix} (of $\gamma$).\vspace{3mm}

To determine how the normalized potential behaves under the action of symmetries we perform a Birkhoff decomposition of $F_{\lambda}=F_-F_+$ (away from some discrete subset) and obtain:

\begin{theorem}\label{thm-symF_-}
Let $\D$ be contractible and $f:\D \rightarrow SO^+(1,n+3)/ SO^+(1,3)\times SO(n)$ be a harmonic map with symmetry
$(\gamma , R)$. Then for  $F_-$ we have
\begin{equation} \label{symF_-}
\gamma^* F_- = \chi  F_- V_+
\end{equation}
for $\chi$ as above and some $V_+ $. Moreover $V_+\in
\Lambda^+ SO^+(1,n+3,\mathbb{C})_{\sigma}$ if $\det\chi=1$, and $V_+\tilde{P}\in
\Lambda^+ SO^+(1,n+3,\mathbb{C})_{\sigma}$ if $\det\chi=-1$.
For the normalized potential $\eta$ of $f$ we obtain
\begin{equation}\label{symmetryeta}
\gamma^* \eta = V_+^{-1} \eta V_+ + V_+^{-1}  dV_+.
\end{equation}
\end{theorem}
\begin{proof}
First, from Theorem \ref{th-symm-1}, we have
$$\gamma^*F_{\lambda}=\chi F_{\lambda} k.$$
Therefore
$$\gamma^*F_-=\chi F_{\lambda}k (\gamma^*F_+)^{-1}=\chi F_-F_+k(\gamma^*F_+)^{-1}=\chi F_-V_+,
$$with
$$ ~ V_+=F_+ k (\gamma^*F_+)^{-1}\in \Lambda ^+SO^+(1,n+3,\C)_{\sigma} \hbox{ if } \det k=1,$$
and
$$ ~ V_+\tilde{P}\in \Lambda ^+SO^+(1,n+3,\C)_{\sigma} \hbox{ if } \det k=-1.$$
As a consequence, we obtain
$$\gamma^*\eta=(\gamma^*F_-)^{-1}d(\gamma^*F_-)=  V_+^{-1} \eta V_+ + V_+^{-1}  dV_+.$$
\end{proof}

In general equation (\ref{symmetryeta}) is a very complicated equation.
Even if one uses, where this is possible, a holomorphic potential in place of
the normalized potential, the situation will, in general, be equally complicated.
Hence it is important to know that, by some clever choice of the potential, one can assume $V_+ = I$, if the symmetry belongs to the fundamental group of some surface $M$. Thus we will not miss any examples by starting from
nice potentials. We will discuss this in more detail below.

For our applications it will be important to have the following theorem

\begin{theorem}\label{thm-sym-F}
Let $\D$ be contractible and assume that $\eta$ is a potential for some harmonic map $f:\D\rightarrow SO^+(1,n+3)/ SO^+(1,3)\times SO(n)$.
Assume moreover that $\gamma$ is a conformal automorphism of $\D$
and that equation \eqref{symmetryeta}
holds for some $V_+ \in  \Lambda ^+SO^+(1,n+3,\C)_{\sigma})^0$ (or $ \tilde{P} V_+\in  \Lambda ^+SO^+(1,n+3,\C)_{\sigma})^0$).
Then  there exists some $\rho (\lambda ) \in
(\Lambda SO^+(1,n+3,\C)_{\sigma})^0$ (or $\rho (\lambda ) \tilde{P}\in
(\Lambda SO^+(1,n+3,\C)_{\sigma})^0$) such that for the solution $C$ to
 the equation $$dC = C \eta,\ ~~C (z=0,\lambda ) = I$$ satisfies
\begin{equation}\label{symmfor-C} \gamma^*C = \rho C V_+.
\end{equation}

Moreover, $\gamma$ induces a symmetry of the harmonic map $f$ associated with $\eta$ if and only if $V_+$ and $\rho$ as above can be chosen such that
 $\rho(\lambda) \in (\Lambda SO^+(1,n+3)_{\sigma})^0$ (or $\rho(\lambda)\tilde{P} \in (\Lambda SO^+(1,n+3)_{\sigma})^0$) and \eqref{symmfor-C} holds.
In this case $\gamma$ induces the symmetry
\begin{equation}\label{symmfor-f}
\gamma^* f = \rho(\lambda) f
\end{equation}
of the harmonic map $f$.
\end{theorem}

\begin{proof}
Since $\gamma^*C$ and $CV_+$ satisfy the same ODE by (the proof of ) \eqref{symF_-},
they only differ by some matrix $\rho \in \Lambda SO^+(1,n+3,\C)_{\sigma}$ which does not depend on $z$ nor $\bar{z}$.
Since $C$, $\gamma^*C$ and $V_+$ are contained in the connected component of our loop group,  also the matrix $\rho(\lambda)$ is contained in the this group. We have seen in Theorem \ref{th-symm-1} that $\rho \in
 (\Lambda SO^+(1,n+3)_{\sigma})^0$, if $\gamma$ is associated with a symmetry. Conversely, if we can choose $\rho$ to be in this group, then in the Iwasawa decomposition $ F = C W_+$ we obtain
$\gamma^*F = \rho F L_+$,
where, by our assumption,  $L_+\in  (\Lambda SO^+(1,n+3)_{\sigma})^0$. Hence
$L_+$ is independent of $\lambda$ and contained in $K$. Now
$\gamma^*f = \rho f$ follows.
\end{proof}

\begin{remark}
The choice of $\rho$ above is in some cases not unique, since it can happen that there exists some $\beta (\lambda)$ such that  $\beta F = F S_+$
 ( non-trivial "isotropy of the dressing action" ). In this case one can choose $\rho$, but equally well $\rho \beta$, as monodromy matrix.
\end{remark}

So far we have only considered harmonic maps defined on contractible domains. But it is easy to extend the discussion to Riemann surfaces $M$ which are either non-compact or of positive genus.

The following result can be proven as in \cite{DPW}, \cite{Do-Ha}, \cite{Do-Ha2}.
\begin{theorem}
Let $M$ be a Riemann surface which is either non-compact or compact of positive genus.

(1) Let $\mathcal{F}:M \rightarrow G/K$ be a harmonic map and $\tilde{\mathcal{F}}$ its lift to the universal cover $\tilde{M}$. Then there exists a normalized potential and a holomorphic potential for $\mathcal{F}$, namely the corresponding potentials for $\tilde{\mathcal{F}}$. Moreover, these potentials satisfy (\ref{symmetryeta}) for every $g \in \pi_1 (M)$.

(2) Conversely, starting from some potential producing a harmonic map
$\tilde{\mathcal{F}}$ from $\tilde{M}$ to $G/K$  and satisfying
(\ref{symmetryeta}) for every $\gamma \in \pi_1 (M)$,
one obtains a harmonic map $\mathcal{F}$ on $M$ if and only if

(2a) The monodromy matrices $\chi(g, \lambda)$ associated with
$g \in \pi_1 (M)$, considered as automorphisms of $\tilde{M}$, are elements of $(\Lambda G_{\sigma})^0$.

(2b) There exists some $\lambda_0 \in S^1$ such that $\chi(g,\lambda_0) = I$ for  all $g \in \pi_1 (M)$, i.e.
\begin{equation}
\begin{split}
F(g.z,\overline{g.z},\lambda=\lambda_0)&=\chi(g, \lambda  = \lambda_0) F(z, \bar{z}, \lambda = \lambda_0)\\
&=  F(z, \bar{z}, \lambda = \lambda_0)  \mbox{mod} \ K
\end{split}
\end{equation} for all $g \in \pi_1 (M)$.
\end{theorem}

 Note that in these cases, $\det\chi$ will always be $1$.

To include the case $M = S^2$ we note that the discussion of symmetries carried out above also applies to harmonic maps from $S^2$ to our symmetric space $G/K$, since we know
\begin{theorem}\label{th-potential-sphere} \cite{DoWa1}
Every harmonic map from $S^2$ to any inner symmetric space $G/K$
can be obtained from some meromorphic normalized potential.
\end{theorem}

\begin{remark} When considering an orientation preserving symmetry of a harmonic map $f: S^2 \rightarrow G/K$, $ f(\gamma.z) = Rf(z)$, we know that the holomorphic automorphism $\gamma$ of $S^2$ has a fixed point $z_0$.
Hence we can consider $\C \cong S^2 \setminus \lbrace z_0 \rbrace$ and apply results obtained for non-compact domains.

From  Theorem \ref{thm-induce-sym} we know that  a Willmore two-sphere $y$ with symmetry $R$ induces a conformal automorphism $\gamma$ of $S^2$ and the remark just above applies.

If one has a harmonic map from $\C$, how can one say that whether does it come from a harmonic map from $S^2$ or not? First we note that from Theorem \ref{th-potential-sphere} and the proof in \cite{DoWa1}, one will see that the holomorphic functions of the normalized potential $\eta$ are in fact all rational functions (on $S^2$), and the integration of $dF_-=F_-\eta$, $F_-(z_0,\lambda)=I$, are also meromorphic on $S^2$. These two conditions are exactly the condition to ensure that the harmonic map is defined on $S^2$.
\end{remark}
\vspace{2mm}

Finally,  in the context of Willmore surfaces the discussion above only concerns symmetries of the harmonic conformal Gauss map of some Willmore immersion. But this induces directly a statement about the symmetries of Willmore surfaces.

\begin{theorem}\label{th-sym-y}
With the notation of the previous theorem, if the conformally harmonic map $f$ induces a unique Willmore surface $y=[Y]$ into $S^{n+2}$ (i.e., $y$ is either not S-Willmore or  the dual surface of $y$ reduces to a point), then equation
\eqref{symmfor-f} induces the symmetry,
\begin{equation}
\gamma^* y =[\gamma^*Y]= [\chi(\lambda) Y].
\end{equation}
If $f$ induces a pair of dual Willmore surfaces $y=[Y]$ and $\hat{y}=[\hat{Y}]$ into $S^{n+2}$
then equation
\eqref{symmfor-f} induces the symmetry
\begin{equation}
\gamma^* y: =[\gamma^*Y]= [\chi(\lambda) Y],
\ \hbox{ or }\ \gamma^* \hat{y}: =[\gamma^*Y] =[\chi(\lambda) Y].
\end{equation}
\end{theorem}
\begin{proof}  We only need to prove the second claim. So let $Y$ be a lift of $y$ with $f$ as its conformal Gauss map. Let $\hat{y}=[\hat{Y}]$ denote the (non-degenerate) dual surface of $y$.
Let $F$ be a local lift of $f$ with its Maurer-Cartan form $\alpha=F^{-1}dF$ of the form stated in Proposition 2.2 of \cite{DoWa1}. Then we obtain $\pi_0(F)=[Y]=y$ from  Proposition 2.2 of \cite{DoWa1}. Now from the proof of Theorem \ref{thm-symF_-}, one derives that $\gamma^*F=\chi(\lambda)Fk$ for some $k=k(z,\bar{z})$. As a consequence,
$$(\gamma^*(Fk^{-1}))^{-1}d(\gamma^*(Fk^{-1}))=
(\chi(\lambda)F)^{-1}d(\chi(\lambda)F)=F^{-1}dF=\alpha.$$
 Therefore, since $$ f = F\mod K = (Fk^{-1}) \mod K,$$ the map $\gamma^*f$ is the conformal Gauss map of $\pi_0({\gamma^*(Fk^{-1}))=}\pi_0(\chi(\lambda)F)
=[\chi(\lambda)Y]=\chi(\lambda)y$.
On the other hand, $\gamma^*f$ is the conformal Gauss map of $\gamma^*y$ and also of $\gamma^*\hat{y}$  and since there exist only two  Willmore surfaces with the same conformal Gauss map (\cite{Ejiri1988}, \cite{Ma2005}), the claim follows.
\end{proof}


\subsection{Willmore immersions admitting finite order symmetries}

The incorporation of symmetries into the loop group formalism is
generally not easy. However, there are a few cases where this is relatively easy
to accomplish. In this subsection we consider Willmore immersions and associated (conformally) harmonic maps defined on some simply connected Riemann surface $\tilde{M}$ and assume that we have some symmetry
$(\gamma, R)$  of finite order, i.e. we assume $\gamma^n = Id$ and $R^n = I$ for some positive integer $n$. Moreover, in all cases $\gamma$ has a fixed point in $\tilde{M}.$

We claim
\begin{theorem}\label{thm-sym-rot}
Let  $y: \tilde{M} \rightarrow S^{n+2}$ be a Willmore immersion
and let $(\gamma,R)$ be a symmetry of $y$ of finite order.
Let $z_0$ denote a fixed point of $\gamma$. Then there exists an extended frame $F$, normalized to $F(z_0,\lambda) = I$, of the conformal Gauss map associated with $y$ such that for all $z \in \tilde{M}$ we have the identity
\begin{equation}\label{rot-F}
F(\gamma.z, \lambda) = T F(z,\lambda)T^{-1},
\end{equation}
 where $T \in K \subset SO^+(1, n+3)$ if $\det R=1 $ and $T\tilde{P} \in K \subset SO^+(1, n+3)$ if $\det R=-1$.
Moreover, if $ F = F_- L_+$ denotes the Birkhoff splitting of $F$, then we have
\begin{equation}\label{rot-F_-}
F_-(\gamma.z, \lambda) = T F_-(z,\lambda)T^{-1}
\end{equation}
and for the Maurer-Cartan form $\eta$ of $F_-$, i.e. the normalized potential of the conformal Gauss map of $y$, we obtain
\begin{equation}\label{rot-eta}
\eta(\gamma.z, \lambda) = T \eta(z,\lambda)T^{-1}.
\end{equation}
Conversely, if we start from some normalized potential $\eta$ satisfying
(\ref{rot-eta}) for some finite order symmetry $(\gamma,R)$, then the solution to the ode $dC = C \eta$, $C(z_0,\lambda) = I$, where $z_0$ denotes a fixed point of $\gamma$, satisfies (\ref{rot-F_-}) and the corresponding frame $F$, obtained from $C$ by the (locally) unique Iwasawa splitting,  satisfies (\ref{rot-F}).
From this we obtain
\begin{equation}
f(\gamma.z,\lambda) = T f(z,\lambda)
\end{equation}
and, if $f$ induces a Willmore immersion $y$, then $y$ inherits the symmetry  in Theorem \ref{th-sym-y}.
\end{theorem}

\begin{proof}
By  Theorem \ref{th-symm-1} we know $F(\gamma.z,\lambda) = \chi (\gamma, \lambda) F(z, \lambda) k (\gamma,z)$
 with $\chi (\gamma, \lambda) \in \Lambda SO^+(1, n+3)_\sigma^\circ$ if $\det R=1$ and  $(\chi (\gamma, \lambda)\cdot\tilde{P}) \in \Lambda SO^+(1, n+3)_\sigma^\circ$ if $\det R=-1$.

Evaluating at $z = z_0$ which we choose to be $0$, we obtain
$$I = F(0,\lambda) = F(\gamma.0,\lambda) = \chi (\gamma, \lambda) F(0, \lambda) k(\gamma,0) = \chi (\gamma, \lambda) k(\gamma,0) .$$
This shows
$$\chi (\gamma,\lambda) = k(\gamma,0)^{-1} =T$$
and implies that $\chi$ is independent of $z$ and of $\lambda$.
Performing a Birkhoff decomposition $ F = F_- W_+$ we obtain  $$F_- (\gamma.z,\lambda) W_+(\gamma.z,\lambda) = F(\gamma.z,\lambda)  = T F(z,\lambda) k(\gamma,z) = T F_-(z,\lambda)T^{-1} \cdot
T W_+(z,\lambda) k(\gamma,z)$$ from which, together with the initial condition, we infer
\begin{equation*}
F_-(\gamma.z, \lambda) = T F_-(z,\lambda)T^{-1},\ \hbox{ and }\ W_+(\gamma.z,\lambda) = T W_+(z,\lambda) k(\gamma,z).
\end{equation*}

Splitting off the leading term $W_0$ of $W_+$ we obtain $W_+ = W_{++} W_0$.
The equation above now yields the relations $$ W_{++} \circ \gamma =
TW_{++} T^{-1}\ \hbox{ and }\ W_0 \circ \gamma = T W_0 k.$$
Writing $$W_0= \hbox{diag}(W_1,W_2) \in SO(1,3,\C) \times O(n,\C),$$ we can decompose
$ W_0 = s \hat{W}_0$ with $s= \hbox{diag}(s_1,s_2) \in SO(1,3,\C) \times O(n,\C)$ and
$\hat{W}_1= \hbox{diag}(\hat{W}_1,\hat{W}_2) \in SO(1,3,\C) \times O(n,\C)$ such that $\hat{W}_1 \in SO(1,3)$,
$\hat{W}_2 \in O(n)$, $s_1 \in \mathcal{ S}$ (see Theorem 4.5 in \cite{DoWa1}) and $s_2$ in some Borel subgroup of $O(n,\C)$. Note that now
 $W_1 = s_1 \hat{W}_1 $ and $ W_2 = s_2 \hat{W}_2$.
Since  $T$ is of the form $ T = (T_1,T_2)$ and since conjugation by $T$ leaves the group $S$ of  Theorem 4.5 in \cite{DoWa1} invariant,
we conclude further $$s \circ \gamma = TsT^{-1}\ \hbox{ and }\ \hat{W}_0 \circ \gamma = T\hat{W}_0 k.$$
So $$\hat{W}_0^{-1} \circ \gamma = (T\hat{W}_0 k)^{-1}=k^{-1}\hat{W}_0^{-1}T^{-1}.$$
Next we replace $F$ with $\hat{F} =
F \hat{W_0}^{-1} $ and obtain by a direct computation
$$\hat{F} \circ \gamma =(F\circ \gamma )\cdot( \hat{W_0}^{-1}\circ \gamma)=T F k\cdot k^{-1} \hat{W_0}^{-1}T^{-1}  = T \hat{F} T^{-1}.$$
Checking the construction above step by step one observes that $\hat{F}(0,\lambda) = I$ holds.

For the converse it suffices to split $C = F W_+$ locally such that $W_0 \in S$, which is possible by Theorem 4.5 in \cite{DoWa1}. Hence locally we obtain
 $\gamma^*F=TFT^{-1}$. Since $f$ is real analytic, we see that  $\gamma^*F=TFT^{-1}$ holds globally.
\end{proof}

\begin{remark}
Most Willmore surfaces do not have any symmetries at all. We have investigated the case of finite order symmetries above. Of course, there  also
exist Willmore surfaces which have symmetries of infinite order. Even 1-parameter groups of symmetries occur (''equivariant Willmore surfaces'') as well as the cases of two-dimensional and three-dimensional symmetry groups.
These cases need a separate treatment. But just as an illustration of what can happen we list below an equivariant example.
\end{remark}


\subsection{Some examples}

\begin{example}\label{ex-1}  Let  $$\eta=\lambda^{-1}\left(
                      \begin{array}{cc}
                        0 & \hat{B}_1 \\
                        -\hat{B}_1^tI_{1,3} & 0 \\
                      \end{array}
                    \right)dz,\ \hbox{ with }\ \hat{B}_1=\frac{1}{2}\left(
                     \begin{array}{cccc}
                       -i&  1 \\
                       i&  -1  \\
                       -z & -iz   \\
                       iz & -z    \\
                     \end{array}
                   \right).$$
That is, $f_1=\frac{z^2}{2}, f_2=z,$ and $f_3=f_4=0$, in Theorem \ref{th-willmore-iso-formula}. By \eqref{eq-willmore in s4-y-1}, we see that  The corresponding associated family of Willmore surfaces is ($r=|z|$)
\begin{equation}\label{example2} y_\lambda =\frac{1}{1+\frac{1}{4r^2}+ \frac{r^2}{4}}
  \left( 1-\frac{1}{4r^2}- \frac{r^2}{4} ,\
  \frac{i(z-\bar{z})}{2r^2} ,\ \frac{z+\bar{z}}{ 2r^2} ,\  -\frac{i(\lambda^{-1}z-\lambda \bar{z})}{2},  \frac{\lambda^{-1}z+\lambda \bar{z} }{2}
\right)^t.
                   \end{equation}
For each $\lambda\in S^1$, $y_\lambda$ is an embedded Willmore sphere in $S^4$ with Willmore energy $W(y_{\lambda})=4\pi$ and  $y_{\lambda}$ is conformally equivalent to the minimal graph  \begin{equation}x_\lambda= \left(
  \frac{i(z-\bar{z})}{2r^2} ,\frac{z+\bar{z}}{ 2r^2} ,\  -\frac{i(\lambda^{-1}z-\lambda \bar{z})}{2},  \frac{\lambda^{-1}z+\lambda \bar{z} }{2}
\right)^t \end{equation}
                   in $\mathbb{R}^4$.

Note that $\eta$  has a symmetry $(\check{\gamma},T)$  of order two $$\eta(\check{\gamma}.z):=\eta(-z)=T\eta T^{-1}.$$
                   with $T=\hbox{diag}\{1,1,-1,-1,-1,-1\}$ and $\check{\gamma}(z) = -z$.
                  The symmetry $(\check{\gamma}, T)$ yields
$$F(\check{\gamma}.z,\lambda)=TF(z,\lambda)T^{-1}.$$
Deriving  $y_{\lambda}$ from $F$ as in (8) of \cite{DoWa1}  we obtain
$$y_{\lambda}(\check\gamma.z)=T_1y_{\lambda},$$
with $T_1=\hbox{diag}\{1,-1,-1,-1,-1\}$.
\end{example}

\begin{example} (\cite{DoWa1}, Theorem 5.12) Let $$\eta=\lambda^{-1}\left(
                      \begin{array}{cc}
                        0 & \hat{B}_1 \\
                        -\hat{B}_1^tI_{1,3} & 0 \\
                      \end{array}
                    \right)dz,\ ~ \hbox{ with } ~\ \hat{B}_1=\frac{1}{2}\left(
                     \begin{array}{cccc}
                       2iz&  -2z & -i & 1 \\
                       -2iz&  2z & -i & 1 \\
                       -2 & -2i & -z & -iz  \\
                       2i & -2 & -iz & z  \\
                     \end{array}
                   \right).$$
The associated family of Willmore two-spheres $x_{\lambda}$, $\lambda\in S^1$, corresponding to $\eta$, is \begin{equation}\label{example1}
  x_{\lambda}=\frac{1}{ \left(1+r^2+\frac{5r^4}{4}+\frac{4r^6}{9}+\frac{r^8}{36}\right)}
\left(
                          \begin{array}{c}
                            \left(1-r^2-\frac{3r^4}{4}+\frac{4r^6}{9}-\frac{r^8}{36}\right) \\
                            -i\left(z- \bar{z})(1+\frac{r^6}{9})\right) \\
                            \left(z+\bar{z})(1+\frac{r^6}{9})\right) \\
                            -i\left((\lambda^{-1}z^2-\lambda \bar{z}^2)(1-\frac{r^4}{12})\right) \\
                            \left((\lambda^{-1}z^2+\lambda \bar{z}^2)(1-\frac{r^4}{12})\right) \\
                            -i\frac{r^2}{2}(\lambda^{-1}z-\lambda \bar{z})(1+\frac{4r^2}{3}) \\
                            \frac{r^2}{2} (\lambda^{-1}z+\lambda \bar{z})(1+\frac{4r^2}{3})  \\
                          \end{array}
                        \right),\ \ r=|z|.
\end{equation}
$x_{\lambda}$ is a Willmore sphere in $S^6$, which is non S-Willmore, full, and totally isotropic.

                 $\eta$ has a symmetry $( \check{\gamma},T)$ of order two $$\eta(\check{\gamma}.z):=\eta(-z)=T\eta T^{-1}.$$
                   with $T=\hbox{diag}\{1,1,-1,-1,1,1,-1,-1\}$ and $\check{\gamma}(z) = -z$.
The symmetry $(\check{\gamma},T)$ yields
$$F(\check{\gamma}.z,\lambda)=TF(z,\lambda)T^{-1}.$$
Deriving  $x_{\lambda}$ from $F$ as in  (8) of \cite{DoWa1} we obtain
$$x_{\lambda}(\check\gamma.z)=T_1x_{\lambda},$$
with $T_1=\hbox{diag}\{1,-1,-1,1,1,-1,-1\}$.
\end{example}

The examples above inherit more symmetries. In fact they admit a one parameter group symmetries, which will be discussed as equivariant Willmore surfaces in other publications. Below we will see a Willmore surface with a three fold symmetry ( and not equivariant).
\begin{example}
Set
\begin{equation}\label{eq-sym-three-order}
    f_1= -\frac{z^3}{3},\ f_2=z^4(1+z^3),\ f_3=z(1+z^3), \ f_4=2z^2+\frac{23z^5}{5}+\frac{7}{2}z^8.
\end{equation}
It is not hard to verify that $f_1'f_4'+f_2'f_3'=0$. Substituting in \eqref{eq-willmore in s4-y-1}, we will obtain a Willmore surface of three fold symmetry. To be concrete, let
\begin{equation}\label{eq-sym-three-order-T}
    \check\gamma(z)=e^{i\theta_1}z,\ \hbox{ and }\ T_{\theta}=\left(
                                                       \begin{array}{cccccc}
                                                         1 & 0 &0& 0& 0 & 0 \\
                                                         0 & 1 & 0& 0& 0 & 0 \\
                                                         0 & 0 &\cos\theta_1& \sin\theta_1& 0 & 0 \\
                                                         0 & 0 &-\sin\theta_1& \cos\theta_1& 0 & 0 \\
                                                         0 & 0& 0 & 0 &\cos\theta_1& \sin\theta_1 \\
                                                         0 & 0 & 0 & 0&-\sin\theta_1& \cos\theta_1 \\
                                                       \end{array}
                                                     \right), \hbox{ with } \theta_1=\frac{2\pi}{3}.
\end{equation}
It is easy to see that
$$\eta(\check\gamma.z)= T_{\theta}\eta(z) T_{\theta}^{-1}.$$
As a consequence, $\check\gamma$ induces a symmetry on $Y_1$ as $\check\gamma^*Y_1= T_{\theta}Y_1$.
\end{example}
\begin{remark}
  The examples above show that it is possible to compute at least some examples quite explicitly. The general picture has, of course, more parameters and needs more computation.
\end{remark}

 The example below shows that it can happen that for $R\in O^+(1,n+3)$ with $\det R=-1$ there exists some Willmore immersion $y$ satisfying
$R y(M)=y(M)$ such that on $M$ the transformation $R$ induces an orientation preserving automorphism.

\begin{example}
Let
\begin{equation}\label{eq-example R=-1}
    x=Re\left(\frac{z^2}{2}-\frac{z^4}{12},\frac{iz^2}{2}+\frac{iz^4}{12},\frac{z^3}{3},-\frac{iz}{2}+\frac{iz^5}{30},\frac{z}{2}+\frac{z^5}{30}\right)^t.
\end{equation}
It is a minimal surface in $\R^5$ and hence a Willmore surface in $S^5$.
Set $$\gamma(z)=-z, \hspace{2mm} \mbox{and} \hspace{2mm} \check R_1=\hbox{diag}(1,1,-1,-1,-1).
$$
Then we have $$x(\gamma(z))=\check R_1x,\hspace{2mm} \mbox{and} \hspace{2mm}  \det\check R_1=-1.$$
Moreover, $x$ has the lift
$$Y=\left(\frac{1+|x|^2}{2},\frac{1-|x|^2}{2},x\right).$$
It is easy to see that we obtain
$$\gamma^*Y=\check{R}\cdot Y\ \hbox{ with }\ \check{R}=\hbox{diag}(1,1,1,1,-1,-1,-1), \hspace{2mm} \mbox{and} \hspace{2mm} \det \check R=-1.$$
\end{example}

\begin{example}
Let's consider the differential 1-form  $\eta = D dz$  on $\C,$ where
\begin{equation}
D = \lambda{-1}D_{-1} + D_0 + \lambda D_1 \in \Lambda so(1,3+n)_\sigma.
\end{equation}
Then $\eta$ is a holomorphic potential of some harmonic map into $G/K$ with
$G = SO^+(1,n+3)$ and $K = SO^+(1,3) \times SO(n).$
This harmonic map will  be associated with a Willmore  surface if $D(\lambda = 1)$ has the form given for $\alpha$ in Proposition 2.2 of \cite{DoWa1} with the additional property that for the coefficients $D_{0,13}$ and
$D_{0,23}$ of $D_0$ one has $D_{0,13} + D_{0,23} \neq 0.$

Putting $C(z,\lambda) = \exp(zD)$ and $\chi(t,\lambda) =\exp(tD)$ we obtain trivially  $C(z,\lambda) = I$ the equation
\begin{equation}
C(z+t,\lambda) = \chi(t,\lambda) C(z,\lambda)
\end{equation}

Since we have assumed that $D$ is real,
$D  \in \Lambda so(1,3+n)_\sigma,$ it is obvious that $\chi(t,\lambda) \in
\Lambda SO^+(1,n+3)_\sigma.$ This implies that
the Willmore surface $y_\lambda$ associated with $\eta$ has the symmetries
$( g_t, \chi(t, \lambda), \hspace{1mm} t \in \R$.
In particular, for the surface $y$ we obtain the
1-parameter group $(g_t, \exp(tD(\lambda=1))$ of symmetries.

For CMC-surfaces in $\R^3$ one obtains this way the associated family of Delaunay surfaces. We are therefore interested in closing conditions.

\begin{theorem}\cite{DoWa-equ}
Consider the differential 1-form  $\eta = D dz$  on $\C,$ where
\begin{equation}
D = \lambda^{-1}D_{-1} + D_0 + \lambda D_1 \in \Lambda so(1,3+n)_\sigma
\end{equation}
has the same form as the Maurer-Cartan form $\alpha_{\lambda}$ in Proposition 2.2 of \cite{DoWa1},
 and where the coefficients $D_{0,13}$ and
$D_{0,23}$ of $D_0$ satisfy $D_{0,13} + D_{0,23} \neq 0.$

Then

(1) The potential $\eta$ induces a Willmore surface $y$ into $S^{n+2}$

(2) The Willmore surface $y$  has the one-parameter group of symmetries,
$( g_t, \exp(tD(\lambda = 1))$,
where $g_t(z) = z + t$. In particular, $y$ is an equivariant Willmore surface.

(3) If $\exp(2 \pi D(\lambda = 1) )= I$, then the Willmore surface $y$ descends to a
Willmore surface from the cylinder $\C/2 \pi \mathbb{Z}$ to $S^{n+2}.$
\end{theorem}

Note that the continuous family of symmetries contains, in particular,
also symmetries of infinite order.
\end{example}

\begin{remark}
The natural generalization of the usual use of notation  would call the examples  above ``equivariant cylinders''.
It would be very interesting to determine, similar to many other surface classes, all equivariant tori.

\end{remark}


\section{Willmore immersions from surfaces with non-trivial fundamental group}

\subsection{Symmetries induced by the fundamental group of a Riemann surface}

So far we have discussed how to construct Willmore surfaces, defined on a simply-connected domain, which admit some symmetry.
This symmetry generally shows up in every element of the associated family.

A particularly interesting case is given by the groups of symmetries induced by the fundamental group of some Riemann surface.

Let $M$ be a Riemann surface, different from $S^2$, and let $\pi_1(M)$ denote the fundamental group of $M$. Since we exclude in this subsection the case $M = S^2$ we
know that the universal cover $\D = \tilde{M}$ of $M$ is a contractible open subset of $\C$ and $\pi_1(M)$ acts on $\D$ by Moebius transformations.

Then $\Gamma = \{ (\gamma, e), \gamma \in \pi_1(M)\},$ where $e$ is the identity operation on $S^{n+2},$ is for any Willmore surface $y:M \rightarrow G/K$ a group of symmetries for $\tilde{y}: \tilde{M} \rightarrow G/K$ as well as for the corresponding (conformally) harmonic Gauss map $\tilde{f}$.

As a consequence, $\Gamma$ is a group of symmetries for each member
$\tilde{y}_{\lambda}$ and $\tilde{f}_{\lambda}$ of the corresponding associated families. In particular, for every $(\gamma,e) \in \Gamma$ there exists some $\chi(\gamma,\lambda)$ such that

\begin{equation}
\gamma^* \tilde{F}(z,\bar{z},\lambda ) = \chi(\gamma,\lambda) \tilde{F}(z,\bar{z},\lambda )k(z,\bar{z}),
\end{equation}
and
\begin{equation}
\gamma^* \tilde{f}(z,\bar{z},\lambda ) = \chi(\gamma,\lambda) \tilde{f}(z,\bar{z},\lambda ),
\end{equation}
and also
\begin{equation}
\gamma^* \tilde{y}(z,\bar{z},\lambda ) = \chi(\gamma,\lambda) \tilde{y}(z,\bar{z},\lambda ).
\end{equation}
Of course, also the corresponding formula for the transformation behavior of the normalized potential of $\tilde{f}$ holds as well.
In particular, for general $\lambda \in S^1,$ all transformation formulas are as in the general case. However, for $\lambda = 1$ the formulas simplify, since in this case $\chi(\gamma, \lambda = 1) = I$ for every $\gamma \in \pi_1(M)$.
In this case $\tilde{y}$ and $\tilde{f}$ are invariant under the action of $\pi_1(M)$ and $\tilde{F}$ basically also is.

It is natural to ask whether for every Riemann surface $M$ and every harmonic map from $M$ to $G/K$ there does exist some potential $\tilde{\eta} $ on
$\D$ which generates the given harmonic maps and is invariant under the action of $\pi_1 (M)$.

It turns out that this is indeed the case (including the case $M = S^{2}$).
\vspace{2mm}

\begin{theorem}
Let $M$ be a Riemann surface
Let $\tilde{f}:\D = \tilde{ M}\rightarrow G/K = SO^+(1,n+3)/ SO^+(1,3)\times SO(n)$ denote  a harmonic  map. Then

(1) If $M$ is non-compact, then there exists a holomorphic potential $\eta$  on $\D$ generating $\tilde{f}$, whence also generating $f$, which is invariant under the action of $\pi_1(M)$, i.e. $\gamma^* \eta = \eta$ holds for all $\gamma \in \pi_1(M)$.

(2) If $M$ is compact, then there exists a meromorphic potential $\eta$  on $\D$ generating $\tilde{f}$, whence also generating $f$, which is invariant under the action of $\pi_1(M)$, i.e. $\gamma^* \eta = \eta$ holds for all $\gamma \in \pi_1(M)$.
\end{theorem}
\begin{proof}
The proof of (1) is as in \cite{Do-Ha5} and the proof of (2) is in the next section.
\end{proof}

Applying this result to the conformal Gauss map of some Willmore immersion we obtain

\begin{corollary}
Let $M$ be a Riemann surface,  $y: M \rightarrow S^{n+2}$ a Willmore immersion into $S^{n+2}.$ Then

(1) If $M$ is non-compact, then there exists a holomorphic potential $\eta$  on $\D$ which is invariant under the action of $\pi_1(M)$ and generates $y$.

(2) If $M$ is compact, then there exists a meromorphic potential $\eta$  on $\D$  which is invariant under the action of $\pi_1(M)$ and generates $y$.
\end{corollary}
\subsection{Examples}

\begin{example}
It is easy to see that if we set
\begin{equation}\label{eq-meromorphic functions-r3}
df_1=d\mathbf{f},\ df_2=df_3=\mathbf{h}^2d\mathbf{f},\ df_4=\mathbf{h}d\mathbf{f},\end{equation}  the condition
 $f_1'f_4'+f_2'f_3'=0$ will be satisfied automatically. By the main theorem of Kichoon Yang \cite{Yang}, let $\mathbf{f}$ be any nonconstant meromorphic function on a  compact Riemann surface $M$ of positive genus,
 there exists a non-zero meromorphic function $\mathbf{h}$ such that the $1-$forms in \eqref{eq-meromorphic functions-r3} have no residues and no periods on $M\backslash\Sigma$, where $\Sigma$ are the poles of $\mathbf{f}$ and $\mathbf{h}$. As a consequence, $f_1,\ f_2,\ f_3,\ f_4$ are meromorphic functions on $M$, and we obtain a (branched) Willmore surface in $S^4$ by substituting these functions into \eqref{eq-willmore in s4-y-1}, which is conformally equivalent to a minimal surface in $S^4$.

 Moreover, if we set
\begin{equation}\label{eq-meromorphic functions-r3-deform}
df_1=d\mathbf{f},\ df_2=\frac{1+t_0^2}{2t_0}\mathbf{h}^2d\mathbf{f},\ df_3=\frac{i(1-t_0^2)}{2t_0}\mathbf{h}^2d\mathbf{f},\ df_4=\mathbf{g}d\mathbf{f},\end{equation}
for some complex function $t_0\in \C^*$, then the resulting (branched) Willmore surface is globally defined on $M$ and it is non-minimal in any space form when
$\frac{1+t_0^2}{2t_0} $ and $\frac{i(1-t_0^2)}{2t_0}$ are real linearly independent. Note that in \cite{Mon}, such surfaces are called the twistor deformation of the original minimal surfaces (See Corollary 8 and the remarks therein of \cite{Mon}).
\end{example}


\section{Existence of invariant potentials}

An essential tool in the loop group method \cite{DPW} is the use of
potentials. Since there is a gauge freedom in the choice of potentials for a given surface it is not surprising that different types of potentials are particularly useful for different purposes. For example, for the construction of surfaces with certain symmetries, it is particularly useful to start from potentials which already reflect the desired symmetry. Similarly, for the construction of conformally harmonic maps from $\mathcal{F} : M\rightarrow G/K$ is it useful to start from potentials on the universal cover $\tilde{M}$ of $M$ which already come from a differential one-form on M. Thus one would like to start from a potential which is invariant under the action of the fundamental group of $M$
on $\tilde{M}$.

If $M$ is a non-compact Riemann surface it was shown in \cite{Do-Ha2}
that every CMC surface in $\R^3$ can be obtained from a holomorphic potential on  $\tilde{M}$ which is invariant under the action of the fundamental group of $M$ on  $\tilde{M}$.

It has been conjectured for some time that the analogous result holds for compact Riemann surfaces $M$ if one permits the potentials to be meromorphic.

It is the goal of this appendix to prove this conjecture for harmonic maps
$\mathcal{F}:M\rightarrow G/K$, where $G = SO^+(1,m)\subset SL(n,\C)$ for some $n$. Since in this case $G$ is a matrix Lie group, we use
the usual $I$ to denote the identity of $G$ in this section.

It is natural (and necessary for our proof) to distinguish two cases.

\subsection{(A) Harmonic maps $\mathcal{F} : S^2 \rightarrow G/K$.}

In \cite{DoWa1}, we have given a proof for the existence of the normalized potential, whence for $S^2$  this potential is already invariant.

\subsection{(B) Harmonic maps  $\mathcal{F} : M \rightarrow G/K$ with $M$ compact and of positive genus $g$ and $G^\C = SO^+(1+m,\C), m \geq 3 $.}

\subsubsection{The basic setting.}

Let $\mathcal{F}: M \rightarrow G/K$ be a harmonic map and
$\tilde{\mathcal{F}}:\tilde{M}\rightarrow G/K$ its lift to the universal cover $\tilde{M}$ of $M$.

Let $F$ denote (see \cite{DPW} or Section 2 ) a globally defined extended frame  $F:\tilde{M}\rightarrow G$ for $\tilde{\mathcal{F}}$. We will also assume that $F(z,\bar{z},\lambda)$ attains the value $I$ at some fixed base point $z_0$.
So we obtain (by the construction of \cite{Do-Ha2})

\begin{equation} \label{framesym}
F(g.z)=F(z)\mathcal{K}(g,z),\ \hspace{2mm}\mbox{ for all } \hspace{2mm} g\in\pi_1(M)\subset Aut(\tilde{M}).
\end{equation}
Therefore $\mathcal{K}(g,z)$ is a ``crossed homomorphism" with values in $G$, i.e. we have
\begin{equation} \label{crossedhom}
 \mathcal{K}(gh,z)=\mathcal{K}(h,z)\cdot \mathcal{K}(g,h.z).
\end{equation}
Introducing the loop $\lambda$ as usual, we obtain (see e.g. \cite{Do-Ha2})
\begin{equation} \label{extframesym}
F(g.z,\lambda)=\chi(g,\lambda)F(z,\lambda)\mathcal{K}(g,z)
\end{equation}
and since $\mathcal{K}$ is a crossed homomorphism we infer
\begin{equation}
\chi:\pi_1(M)\rightarrow \Lambda G_{\sigma} \hspace{2mm}\mbox{ is a homomorphism}.
\end{equation}
Note that the same crossed homomorphism $\mathcal{K}$ occurs in (\ref{framesym})
and in (\ref{extframesym}) and is independent of $\lambda$.

Let's now consider any  holomorphic extended frame $\tilde{C}$ and the corresponding holomorphic potential  $\tilde \eta$ for $\tilde{\mathcal{F}}$
(see Section 2 of our paper). Thus we have
\begin{equation}
d\tilde{C}=\tilde{C}\tilde{\eta},\ \ ~ \hbox{ with } \tilde{C}(z_0,\lambda)=I.
\end{equation}
Since $\tilde{C} = F W_+$ it is easy to verify that we obtain for every
$g \in \pi _1 (M)$
\begin{equation}\label{eq-tildeC}
\tilde{C}(g.z,\lambda)=\chi(g,\lambda)\tilde{C}(z,\lambda)\tilde{\mathcal{K}}_+(g,z,\lambda).
\end{equation}
Since $\chi$ is a homomorphism from $\pi _1 (M)$ to $\Lambda G_{\sigma}$
it is straightforward to verify that $\tilde{\mathcal{K}}_+ $ is a crossed homomorphism with values in  $\Lambda^+G^{\mathbb{C}}_{\sigma},$ i.e.,
\begin{equation}\label{eq-k}\tilde{\mathcal{K}}_+(gh,z,\lambda)=\tilde{\mathcal{K}}_+(h,z,\lambda)\tilde{\mathcal{K}}_+(g,h.z,\lambda).
\end{equation}

\subsection{The claim}

{\em Claim: } There exists $\tilde{h}_+:\tilde{M}\rightarrow \Lambda^+G^{\mathbb{C}}_{\sigma}$ meromorphic such that
for every $g \in \pi _1 (M)$ we have
\begin{equation} \label{claim}
\tilde{\mathcal{K}}_+(g,z,\lambda)=\tilde{h}_+(z,\lambda)\tilde{h}_+(g.z,\lambda)^{-1}.
\end{equation}

Assume for the moment this claim has been proven. Then we set
$$C (z,\lambda) = \tilde{C} (z,\lambda) \tilde{h}_+(z,\lambda)$$
and obtain $$ C(g.z, \lambda) = \chi(g,\lambda) C(z,\lambda), ~~ \hbox{ for all }~~ g\in
\pi_1(M).$$
This implies in particular that the Maurer-Cartan form $\eta$  of $C$ is invariant under the action of $\pi_1 (M)$. Thus $\eta$ is the desired potential.\\

{\bf Outline of proof of Claim:}\\

The basic idea is to adjust the proof of Theorem 31.2 of \cite{Forster} (also see Exercise 31.1 of \cite{Forster})

Since the proof in \cite{Forster} does not take into account values in $SO^+(1+m,\C)$,
nor loop groups nor even twisted loop groups, we will break down the proof into several steps:\\

{\bf Step 1: } Show that there exist $\tilde{h}_+ \in \Lambda SL(1+m,\C)$
satisfying (\ref{claim}).\\
\vspace{2mm}

{ \bf Step 2:} Show that for $\tilde{\mathcal{K}}_+(g,z,\lambda) \in \Lambda^+SO^+(1+m,\C)$ also
$\tilde{h}_+$ can be assumed to be in  $\Lambda^+SO^+(1+m,\C)$.\\
\vspace{2mm}

{\bf Step 3:} Show that for $\tilde{\mathcal{K}}_+(g,z,\lambda) \in \Lambda^+SO^+(1+m,\C)_{\sigma}$ also
$\tilde{h}_+$ can be assumed to be in  $\Lambda^+SO^+(1+m,\C)_{\sigma}.$\\

\subsection{Proof of Step 1}

We will follow the proof of theorem 31.2 of \cite{Forster}, but we will adjust at several  places to incorporate our setting.

 (1) We start by following  (a) of \cite{Forster}, but since our action of $\pi_1 (M)$ is slightly different from the one in \cite{Forster} we set
\begin{equation}\tilde{\Psi}_i(z)=\tilde{\mathcal{K}}_+(\eta_i(z)^{-1},z, \lambda)^{-1}.
\end{equation}
This expression is holomorphic and we have $\tilde{\Psi}_i\in\Lambda^+SO^+(1,m,\C)_{\sigma}$.
For the definition of $\eta_i(z)$ we refer to p. 215, (28.3), and p. 232, (31.3)  of \cite{Forster}.

(2) With the definition of (1)  it is easy to verify the sequence of equalities:
\begin{equation*}
\begin{split}\tilde{\Psi}_i(g.z)&=\tilde{\mathcal{K}}_+(\eta_i(g.z)^{-1},g.z, \lambda)^{-1}\\
&=\tilde{\mathcal{K}}_+(\eta_i(z)^{-1}g^{-1},g.z,\lambda)^{-1}\\
&=\left(\tilde{\mathcal{K}}_+(g^{-1},g.z,\lambda)\tilde{\mathcal{K}}_+(\eta_i(z)^{-1},z,\lambda)\right)^{-1}\\
&=\tilde{\Psi}_i(z)\tilde{\mathcal{K}}_+(g^{-1},g.z,\lambda)^{-1}\\
&=\tilde{\Psi}_i(z)\tilde{\mathcal{K}}_+(g,z,\lambda),\\
\end{split}\end{equation*}
where the third equality comes from the identity \eqref{eq-k} and the last equality is because of the fact that
$$I =\tilde{\mathcal{K}}_+(e,z,\lambda)=\tilde{\mathcal{K}}_+(g,z,\lambda)\tilde{\mathcal{K}}_+(g^{-1},g.z,\lambda).$$
Thus (like (b) of \cite{Forster}) $\tilde\Psi_i$ satisfies the desired behavior on $\mathcal{Y}_i=p^{-1}(\mathcal{U}_i)$. See \cite{Forster} for more details on the notation.

(3) As a consequence of (2) we see that the matrices $$\tilde{g}_{ij}=\tilde\Psi_i\tilde\Psi_j^{-1}\in \Lambda^+GL(n,\mathcal{O}(\mathcal{Y}_i\cap\mathcal{Y}_j))$$ are invariant under the action  of $\pi_1(M)$.
Thus these matrix functions  descend to $M$: $$g_{ij}\in \Lambda^+GL(n,\mathcal{O}(\mathcal{U}_i\cap\mathcal{U}_j)).$$
At this point we would like to write the cocycle $\lbrace g_{ij}\rbrace$
as a boundary. But opposite to \cite{Forster} our surface is compact and our functions also depend holomorphically on a parameter $\lambda \in \C$.
But Proposition 3.12 of H. R\"{o}hrl \cite{Rohrl} shows
\begin{equation}
g_{ij}=p_i^{-1}p_j,\ \ ~p_j\in GL(n,\mathcal{O}(\mathcal{U}_j)),
\end{equation}
with maps $p_j$ which are meromorphic in $z$ and holomorphic in
$\lambda \in \C$, in particular, $p_j\in\Lambda^+GL(m,\C)$ for all $j$.

Since all $g_{ij}$ have determinant $1$, the functions $p_j \cdot diag(1,...1,det(p_j)^{-1})$ are also a boundary for the cocycle
$\lbrace g_{ij}\rbrace$. Note that $\det(p_j)$ and $\det(p_j^{-1})$ only depend on $\lambda^k$, $k\geq0$, in their Fourier expansions.

We can thus assume
\begin{equation}
g_{ij}=p_i^{-1}p_j,\ \ ~p_j\in \Lambda^+SL(n,\mathcal{O}(\mathcal{U}_j),
\end{equation}
with maps $p_j$ which are meromorphic in $z$ and even holomorphic in
$\lambda \in \C$.

 Set $\tilde{p}_j=p_j\circ \pi$ with $\pi:\tilde{M}\rightarrow M$ the natural projection.
Then $$\tilde\Phi_j=\tilde{p}_j\tilde{\Psi}_j\in \Lambda^+SL(n,\C)$$ satisfies
\begin{equation}\tilde{\Phi}_i(g.z)=\tilde{p}_i(g.z)\tilde{\Psi}_i(g.z)=
\tilde{p}_i(z)\tilde{\Psi}_i(z)\tilde{\mathcal{K}}_+(g,z,\lambda)
=\tilde\Phi_i \tilde{\mathcal{K}}_+(g,z,\lambda) ,
\end{equation}
and
\begin{equation}
\tilde{\Phi}_i^{-1}\tilde{\Phi}_j=\tilde{\Psi}_i^{-1}\tilde{p}_i^{-1}\tilde{p}_j \tilde{\Psi}_j
=\tilde{\Psi}_i^{-1}\tilde{g}_{ij}\tilde{\Psi}_j=I \hbox{ on }
\mathcal{Y}_i\cap\mathcal{Y}_j,
\end{equation}
in view of $\tilde{g}_{ij}=\tilde\Psi_i\tilde{\Psi}_j^{-1}$ since $\tilde{g}_{ij}$ and $\tilde{p}_i^{-1}\tilde{p}_j$ are invariant under $\pi_1(M)$ and coincide in $\mathcal{U}_i\cap\mathcal{U}_{j}$, whence on $\mathcal{Y}_i\cap\mathcal{Y}_j$.
Therefore there exists a meromorphic map $\tilde{\Phi}:\tilde{M}\rightarrow \Lambda^+SL(n,\mathcal{O}(\tilde{M}))$ with
\begin{equation}\label{eq-89}\tilde{\Phi}(g.z)=\tilde{\Phi}(z)\tilde{\mathcal{K}}_+(g,z,\lambda).
\end{equation}
Hence $h_+(z)=\tilde\Phi(z)^{-1}$ is a matrix function as desired.

Before we continue the proof of the claim we would like to point out an a priori simplification. Considering $\tilde{\mathcal{K}}(g,z,\lambda)\in \Lambda^+SO(1+n,\C)_{\sigma}$ as in \eqref{eq-tildeC}, it is clear that with
$\tilde{\mathcal{K}}(g,z,\lambda)$ also $\tilde{\mathcal{K}}(g,z,\lambda=0)=\tilde{\mathcal{K}}_0(g,z)$ is a crossed homomorphism with values in $K^{\C}$.

Next we consider the $\lambda-$independent term $\delta_0$ of $\tilde{\Phi}.$ We point out that $\delta_0$ does not depend on any $g\in\pi_1(M)$. However, equation \eqref{eq-89}, spelled out for $\tilde{\Phi}$, yields for the  $\lambda-$independent term $\delta_0(z)=\tilde{\Phi}(z,\lambda=0)$ the relation
\begin{equation}\delta_0(g.z)=\delta_0(z)\tilde{\mathcal{K}}_+(g,z,\lambda=0).
\end{equation}
Therefore the $\lambda-$independent term $\delta_0(z)=\tilde{\mathcal{K}}(g,z,\lambda=0)$ is a crossed homomorphism which has the form \eqref{claim}. Therefore, applying the trick pointed out in the beginning of Section 10.3, we observe that
\begin{equation}C^{\sharp}(z,\lambda)=\tilde{C}(z,\lambda)\delta_0(z)^{-1}
\end{equation}
satisfies
\begin{equation}C^{\sharp}(g.z,\lambda)=\chi(g,\lambda)C^{\sharp}(z,\lambda)\tilde{\mathcal{K}}^{\sharp}_{+}(g,z,\lambda)
\end{equation}
where
\begin{equation}\tilde{\mathcal{K}}^{\sharp}_{+}(g,z,\lambda)=I+\mathcal{O}(\lambda).
\end{equation}
As a consequence, from now on, we can always assume w.l.g. $\tilde{\mathcal{K}}_{+}(g,z,\lambda=0)=I$ and $\tilde{\Phi}(z,\lambda=0)$ is invariant under $\pi_1(M)$.


\subsection{Proof of Step 2}

We want to replace the $\tilde\Phi$ above by some
$\tilde{\Phi}_0\in  \Lambda^+ SO(n,\mathbb{C})$.

But solving
$$\tilde{\Phi}(g.z)^{t,-1}=\tilde{\Phi}(z)^{t,-1}\tilde{\mathcal{K}}_+(g,z,\lambda)$$
and
$$\tilde{\Phi}(g.z)=\tilde{\Phi}(z)\tilde{\mathcal{K}}_+(g,z,\lambda)$$
(or, $\tilde{\mathcal{K}}_+(g,z,\lambda)$) implies $$\tilde{\Phi}(z)^{t}\tilde{\Phi}(g.z)^{t,-1}=\tilde{\Phi}(z)^{-1}\tilde{\Phi}(g.z).$$
Therefore we obtain
$$\tilde{\Phi}(z)\tilde{\Phi}(z)^{t}=\tilde{\Phi}(g.z)\tilde{\Phi}(g.z)^{t}.$$
Thus $$\tilde{A}(z)=\tilde{\Phi}(z)\tilde{\Phi}(z)^{t}\in \Lambda^+SL(n,\C)$$
is meromorphic on $\tilde{M}$, symmetric and invariant under $\pi_1(M)$. Moreover, we can assume $\tilde{A}(z_0,\lambda)=I.$ Applying the Gauss algorithm we can write
$$\tilde{A}=\tilde{L}P\tilde{\mathcal{U}}$$
with $\tilde{L}-I$ strictly lower triangular, $P$ a permutation matrix and $\tilde{\mathcal{U}}$ upper triangular and we can assume $\det\tilde{L}=\det P=\det\tilde{\mathcal{U}}=1$. Moreover, we can assume that $P\tilde{\mathcal{U}}P^{-1}$ is still upper triangular and $\tilde{L}$ and $\tilde{\mathcal{U}}$ are meromorphic and invariant under $\pi_1(M)$ with values in $\Lambda^+ SL(n,\mathbb{C})$.

It is known from linear algebra that the condition ``$P\tilde{\mathcal{U}}P^{-1}$ stays upper triangular" makes the representation $\tilde{A}=\tilde{L}P\tilde{\mathcal{U}}$ unique. Therefore $$I=\tilde{A}(z_0)=\tilde{L}(z_0)P\tilde{\mathcal{U}}(z_0)$$ shows $\tilde{L}(z_0)=I,$ $\tilde{\mathcal{U}}(z_0)=I$ and $P=I$.
Moreover, $\tilde{L}$ and $\tilde{U}$ are invariant under $\pi_1 (M)$.

Replacing $\tilde\Phi$ by $$\tilde\Phi_1=\tilde{L}^{-1}\tilde\Phi\in \Lambda^+SL(n,\C)$$ we obtain a meromorphic matrix function which has still the right transformation behavior.
Since $$\tilde{A}_1=\tilde\Phi_1\tilde\Phi_1^{t}=\tilde{L}^{-1}\tilde{\Phi}\tilde{\Phi}^{t}(\tilde{L}^{t})^{-1}=\tilde{\mathcal{U}}(\tilde{L}^t)^{-1},$$ and $\tilde{A}=\tilde{A}^t$, we obtain $\tilde{\mathcal{U}}=D\tilde{L}^t$ with $D$ diagonal and $\tilde\Phi_1$ satisfies $\tilde\Phi_1\tilde\Phi_1^t=D.$ Since $\tilde\Phi_1\in\Lambda^+SL(n,\mathbb{C})$,
 the $\lambda-$independent and $\pi_1(M)$ invariant term $\delta_0$ of $\tilde\Phi_1$ satisfies $\delta_0^2=D_0$. Hence we can assume,
  by considering $\tilde{\Phi}_2=\delta_0^{-1}\tilde{\Phi}_1$, w.l.g $D_0=I$.

 Expanding now
$$D=I+\lambda D_1+\lambda^2 D_2+\cdots,$$
we need to solve $R^2=D$ with some meromorphic $\pi_1(M)$ invariant diagonal matrix $R$, $R(z_0)=I$ and $\det R=1$. Then $\tilde\Phi_0=R^{-1}\tilde\Phi_2$ is the desired solution.

To show the existence of $R$, we need to consider a scalar function of the form  $D = 1+\cdots$. One can write such a function in the form $\exp(a_+(\lambda) )$: Writing $D = 1 + v$, $v = v_1 \lambda + \cdots$, we consider the binomial series
$R = (1 + v)^1/2$  which converges absolutely on the closed unit disk.
This $R$ is in the Wiener algebra and solves, of course, $R^2 = d$.


\subsection{Proof of Step 3}

We need  to show that one can even assume w.l.g. that $\tilde\Phi_0$ is $\sigma-$twisted.

To this end we consider the realization of $G^{\mathbb{C}}$ in the form $G(n,\mathbb{C})$ as defined in (4.1) of \cite{Wang-1}. Then $\sigma=Ad D_0$ with $D_0=\hbox{diag}(-1,-1,I,-1,-1)$ and $\sigma$ leaves invariant the Cartan algebra of diagonal matrices and the Borel (and dual Borel) algebra consisting of upper triangular (lower triangular ) matrices. The involution $\sigma$ leaves the positive and the negative root spaces invariant. Let $\mathcal{X}$ denote the set of negative roots which are contained in $K^{\mathbb{C}}=Fix(\sigma)$. The set $\mathcal{X}$ defines a parabolic subgroup and we obtain from Theorem 3.5.1 of \cite{Do-Gr-Sz}
\begin{equation}\Lambda G(n,\mathbb{C})=\bigcup_{w\in W, \ w \hbox{ is } \mathcal{X}\hbox{-reduced}}B_-wP
\end{equation}
where
$$B_-=\{g\in \Lambda^-G(n,\mathbb{C})|\ g(\lambda=\infty)\in \hbox{dual Borel group}\},\hbox{ and } P=\Lambda K^{\mathbb{C}}\cdot \Lambda^+G(n,\mathbb{C}).$$
Moreover, the restriction ``$\mathcal{X}$-reduced" is equivalent with ``$w$ is  a chosen representation of the set of double cosets". More importantly, for a given coset $B_-wP$ a representation $g=b_-wp$ can w.l.g be made unique, if we require $w^{-1}b_-w\in B_-$. These special conditions we will assume from now on. In particular, the union $(g_0)$ is disjoint. Since $\tilde{\mathcal{K}}_+$ is fixed by $(\hat{\sigma}g)(\lambda)=\sigma(g(-\lambda))$, we obtain
$$\tilde\Phi_0(g.z,\lambda)=\tilde\Phi_0(z,\lambda)\tilde{\mathcal{K}}_+(g,z,\lambda),$$
$$\sigma(\tilde\Phi_0)(g.z,\lambda)=\hat\sigma(\tilde\Phi_0)(z,\lambda)\tilde{\mathcal{K}}_+(g,z,\lambda).$$
Hence
\begin{equation}\tilde\Phi_0(z,\lambda)^{-1}\tilde\Phi_{0}(g.z,\lambda)=(\hat\sigma\tilde\Phi_0)(z,\lambda)^{-1}((\hat\sigma\tilde\Phi_0)(g.z,\lambda)),
\end{equation}
equivalently,
\begin{equation}\label{Phi_0}\tilde\Phi_0(z,\lambda)(\hat\sigma\tilde\Phi_0)(z,\lambda)^{-1}
=\tilde\Phi_{0}(g.z,\lambda)((\hat\sigma\tilde\Phi_0)(g.z,\lambda))^{-1}.
\end{equation}
Writing $\tilde\Phi_0(z,\lambda)=\tilde\Phi_{0,+}(z,\lambda) \tilde{B}_0(z)$ with $\tilde\Phi_{0,+}(z,\lambda)=I+\mathcal{O}(\lambda)$ and $\tilde{B}_0(z)\in G(n,\mathbb{C})$ independent of $\lambda$, we derive from \eqref{Phi_0}
\begin{equation}
S_0(z)=B_0(z)\sigma B_0(z)^{-1}=B_0(g.z)\sigma B_0(g.z)^{-1}.
\end{equation}
Decomposing $S_0$ as above (here we only need the classical finite dimensional theory ) we obtain
\begin{equation}S_0(z)=b_-(z)wp(z).
\end{equation}
Since this representation is unique, we obtain
$$I=S_0(z_0)=b_-(z_0)wp(z_0)$$
with $w=e$, $b_-(z_0)=I,p(z_0)=I$. Then we can write
\begin{equation}S_0(z)=q_-(z)k_0(z)q_+(z),
\end{equation}
where $q_-$, $q_+$ have leading term $I$ and $k_0\in K^{\mathbb{C}}$.
This representation is unique and $q_-$, $k_0$, $q_+$ are meromorphic on $M$ ( rather meromorphic on $\tilde{M}$ and invariant under $\pi_1(M)$).
$$S_0^{-1}=\sigma S_0=\sigma q_-k_0\sigma q_+=q_+^{-1}k_0^{-1}q_-^{-1}$$
implies $q_+=(\sigma q_-)^{-1}$, $k_0=k_0^{-1}$. As a consequence, $\tilde\Phi_3=q_-^{-1}\tilde\Phi_0\in \Lambda^+G(n,\C)$ has the same transformation behavior as $\tilde\Phi_0$ and is meromorphic on $\tilde{M}$, and invariant under $\pi_1(M)$. Moreover
\begin{equation*} \tilde\Phi_3(\hat\sigma\tilde\Phi_3)^{-1}=k_{0}\sigma k_0^{-1}=k_0^2.
\end{equation*}
Put $\tilde\Phi_4=k_0^{-1}\tilde\Phi_3$. Then  $\tilde\Phi_4:\tilde{M}\rightarrow \Lambda^+G(n,\mathbb{C})_{\sigma}$ is meromorphic and satisfies
$$\tilde\Phi_4(g.z,\lambda)=\tilde\Phi_4(z,\lambda)\tilde{\mathcal{K}}_+(g.z,\lambda).$$
As a consequence,
$$\tilde{C}_0(z,\lambda)=\tilde{C}(z,\lambda)\tilde\Phi_4(z,\lambda)^{-1}$$
satisfies (in our original setting, i.e., using $SO^+(1,2m-1,\C)$)
\begin{theorem} $\tilde{C}_0: \tilde{M}\rightarrow \Lambda SO^+(1,2m-1,\mathbb{C})_{\sigma}$ is meromorphic and satisfies
\begin{equation}\tilde{C}_0(g.z,\lambda)=\chi(g,\lambda)\tilde{C}_0(z,\lambda)\end{equation}
for all $g\in\pi_1(M)$, $z\in \tilde{M}$.
\end{theorem}

\begin{corollary} $\eta=\tilde{C}_0^{-1}d\tilde{C}_0$ is a meromorphic $1-$form on $\tilde{M}$ which is invariant under $\pi_1(M)$ and has with $\tilde{C}_0(z,\lambda)$ a meromorphic solution satisfying $\tilde{C}_0(z_0,\lambda)=I$.
\end{corollary}

\begin{corollary} Every harmonic map $\mathcal{F}:M\rightarrow G/K $ can be obtained from

$a)$  a holomorphic potential, if $M$ is non-compact,

$b)$  a meromorphic potential, if $M$ is compact.
\end{corollary}
\begin{corollary} Every harmonic map $\mathcal{F}:M\rightarrow G/K $ can be obtained from

$a)$  a holomorphic extended frame, if $M$ is non-compact,

$b)$  a meromorphic extended frame, if $M$ is compact.
\end{corollary}

\begin{corollary} Every Willmore surface $y:M\rightarrow S^n$ can be obtained from

$a)$  an invariant holomorphic potential, if $M$ is non-compact,

$b)$  an invariant meromorphic potential, if $M$ is compact.
\end{corollary}

\begin{remark} Invariant potentials will contain in general many powers of $\lambda$. Only if the corresponding Willmore
immersion is of finite uniton type only one power of $\lambda$ will occur.\\
\end{remark}

\section{Appendix: Proof of Theorem 3.1}

We consider the contractible domains $\D_1=\D_2=\D\subset\C$ and the complete Willmore immersions $y_1=y_2=y:\D\rightarrow S^{n+2}$. Moreover, let $R$ be a conformal transformation of $S^{n+2}$ satisfying $R y(\D_1)=y(\D _2)$. Our goal is to find some conformal map $\gamma:\D_1\rightarrow \D_2$ such that the diagram
\begin{equation}\label{eq-commute-D}
 \begin{CD}
\D_1@> \gamma >>\D_2\\
@V{y}VV @VV{y}V\\
S^{n+2} @> R>>S^{n+2}\
\end{CD}
\end{equation}
commutes. We choose an arbitrary point $p_0\in\D_1$. Let $U_0$ be an open subset of $\D_1$ with $p_0\in U_0$ such that $y|_{U_0}$ is injective and $R y(U_0)=y(V_0)$, where $V_0$ is an open subset of $\D_2$ such that $y|_{U_0}$ is injective and $R y(p_0)=y(q_0)$ holds for some $q_0\in V_0$. In this way, we obtain a conformal map $\gamma_0:U_0\rightarrow V_0$ such that the diagram
\begin{equation}\label{eq-commute-local}
 \begin{CD}
U_0@> \gamma >>V_0\\
@V{y}VV @VV{y}V\\
S^{n+2} @> R>>S^{n+2}\
\end{CD}
\end{equation}
commutes.

 \begin{corollary}\label{cor-sym-local}
 Let $c:[0,1]\rightarrow \D_1$ be any smooth curve (i.e, actually defined  and smooth on $]-\varepsilon,1+\varepsilon[$ for some $\varepsilon>0$), with $c(0)\subset U_0$. Then there exists some open subset $U$ with $c([0,1])\subset U$, $U_0\subset U$ and a unique conformal map $\gamma: U\rightarrow \D_2$ such that the diagram
\begin{equation}\label{eq-commute-local-2}
 \begin{CD}
U@> \gamma_0 >>\D_2\\
@V{y}VV @VV{y}V\\
S^{n+2} @> R>>S^{n+2}\
\end{CD}
\end{equation}
commutes.
\end{corollary}
\begin{proof}
Let $$I_0=\{t>0;\exists \hbox{ open subset }  U_t\supset c([0,t]) \hbox{ such that \eqref{eq-commute-local-2} commmutes with } U \hbox{ replaced by } U_t \}.$$
Since $c[0,t_0]\subset U_0$ for some $t_0>0$, the set $I_0$ is not empty by setting $\gamma_t=\gamma_0$ and $U_t=U_0$.
Since for $t_0\leq t<t'$ we have $\gamma_t=\gamma_{t'}=\gamma_0$ on $U_0\subset U_t\cap U_{t'}$, by real analyticity we also obtain $\gamma_{t'}(x)=\gamma_t(x)$ for $x\leq t$.

Put $$t_*=\sup I_0.$$
Set $p_*=c(t_*)\in\D_1$ and $v_*=Ry(p_*)\in y(\D_2)$. Let $\{t_j\}\subset (-\varepsilon+t_*,t_*)$ for some $\varepsilon>0$ be a Cauchy sequence converging to $t_*$. Then $\{Ry(c(t_j))\}$ is a Cauchy sequence converging to $v_*$. Since locally $y$ is an isometry from $\D_2$ to $y(\D_2)$, $\{\gamma (c(t_j))\}$ is also a Cauchy sequence in $\D_2$ and has a limit $q_*\in \D_2$ with $y(q_*)=v_*$. It is easy to see that $q_*$ is independent of the choice of $\{t_j\}$ and hence well-defined. Let $V_*\subset\D_2$ be an open subset containing $q_*$ such that $y(V_*)$ is injective.
Since $p_*=c(t_*)\in\D_1$, by shrinking $V_*$ if necessary, there exists some open set $U_*\subset \D_1$ of $p_*$ such that $y|_{U_*}$ is injective, and $Ry(U_*)=y(V_*)$. Therefore we define uniquely a conformal map $\gamma_*:U_*\rightarrow V_*$. For any $t_j$ as above, on $U_{t_j}\cap U_*$ the maps $\gamma_{t_j}$ and $\gamma_*$ coincide. Therefore $\gamma_*$ extends the maps $\gamma_{t_j}$ to the connected component of $\bigcup_{t_j}(U_{t_j}\cup U_*)$ containing $c([0,t_*])$. Hence $t_*=1$.
\end{proof}
\begin{remark}
In the situation of the above corollary, we will say ``$\gamma$ extends $\gamma_0$ along $c$". The proof of Corollary \ref{cor-sym-local} shows that we can assume that all such $\gamma$ are defined on $U_1$, where we replace $U_*$ by $U_1$, since $t_*=1$.
\end{remark}
 Moveover we obtain:
\begin{corollary} If $c_0$ and $c_1$ are homotopic curves connecting $p_0$ and $p_1$ in $\D_1$, then the unique extensions $\hat\gamma_0$ and $\hat\gamma_1$ coinciding on $U_0$  with $\gamma_0$ along $c_0$ and $c_1$ respectively coincide in the neighborhood $\hat U$ of $p_1$.
\end{corollary}

\begin{proof} Without loss of generality we can assume $c_0(0)=c_1(0)=p_0,$ $c_0(1)=c_1(1)=p_1$. Since $c_0$ and $c_1$ are homotopic, there exists a continuous map
$\beta:[0,1]\times[0,1]\rightarrow \D_1$ such that
$$\beta(0,t)=c_0(t), \ \beta(1,t)=c_1(t), \ \beta(s,0)=p_0, \beta(s,1)=p_1.$$
Set $c_s(t)=\beta(s,t)$ and let  $\hat\gamma_s$ be the conformal extension of $\gamma_0$ along $c_s(t)$ which is defined on $U_0$.

By Corollary \ref{cor-sym-local} and the remark above, there exists an open and connected subset $W_0\supset ( U_0\cup c_0([0,1]))$ and an extension $\hat\gamma_0$ of $\gamma_0$ to $W_0$ such that the  diagram
\eqref{eq-commute-local-2} commutes. By continuity, there exists some $\varepsilon>0$ such that $\beta([0,\varepsilon)\times[0,1])\subset W_0$. So for any
fixed $s\in[0,\varepsilon)$, $c_s(t)=\beta(s,t)$ maps $[0,1]$ into $W_0$ and therefore the conformal extension $\hat\gamma_s$ along $c_s(t)$ coincides with
$\hat\gamma_0$ on $W_0$. In particular, $\hat\gamma_s=\hat\gamma_0$ on the open subset $ W_0$ and $p_1\in W_0$. As a consequence, the set
$$\hat I_0=\{s>0; \hat\gamma_s=\hat\gamma_0 \hbox{ on an open subset containing }  p_1\}$$ is not empty.
Set $$s_*=\sup \hat I_0\leq 1.$$
By an argument analogous to the one given just above we see that  there exists some $\varepsilon_*>0$ such that for any
$s\in(-\varepsilon_*+s_*,s_* + \varepsilon_*)$ we have  $\hat\gamma_{s_*}=\hat{\gamma}_{s}$ on an open subset containing $p_1$. Since for any $s\in(-\varepsilon_*+s_*,s_* + \varepsilon_*)$, $\hat\gamma_{s}=\hat{\gamma}_{0}$ on an open subset containing $p_1$, we see that
$\hat\gamma_{s_*}=\hat{\gamma}_{0}$ on an open subset containing $p_1$. Hence $s_*=1$.
\end{proof}

    Finally, we define a map $\gamma:\D_1\rightarrow \D_2$ by fixing $p_0$ and defining $\gamma$ at some point $p\in\D$ by the value of some extension of some curve connecting $p_0$ and $p_-1$. Since $D_1$ is
simply-connected, any two such curves are homotopic in $\D_1$ and thus, by the above corollary, attain the some value at $p_1$. This way, we obtain a well-defined, conformal map $\gamma:\D_1\rightarrow \D_2$ satisfying \eqref{eq-commute-D} as claimed.\\

Note that a similar result is proven in \cite{Helgason} for isometric mappings between complete manifolds (see also \cite{Do-Ha2}). But here we are dealing with conformal maps so that one needs to adjust the proofs in \cite{Helgason} and \cite{Do-Ha2} as above.\\


\ \\
{\bf Acknowledgements}\ \
Part of this work was done when the second named author visited the Department of Mathematics of Technische Universit\"{a}t  M\"{u}nchen. He would like to express his sincere gratitude for both the hospitality and financial support. The second named author was partly supported by the Project 11201340 of NSFC  and
the Fundamental Research Funds for the Central Universities.

{\small

\def\refname{Reference}

}

\end{document}